\documentclass[a4paper,12pt,reqno]{amsart}

\usepackage{amsmath,amssymb,amsthm}
\usepackage{color}
\usepackage{graphicx}
\usepackage{mathrsfs}
\usepackage{enumerate}
\usepackage{mathtools}
\usepackage[abbrev]{amsrefs}
\usepackage[T1]{fontenc}
\usepackage[colorlinks,
           linkcolor=blue,      
           anchorcolor=blue,  
           citecolor=red,       
            ]{hyperref}
\usepackage{amsthm}
\usepackage{amsmath,amssymb,amsthm}
\usepackage{latexsym}
\usepackage{color}
\usepackage{graphicx}
\usepackage{mathrsfs}
\usepackage{enumerate}
\usepackage[abbrev]{amsrefs}
\usepackage[T1]{fontenc}
\usepackage{mathtools}
\mathtoolsset{showonlyrefs=true}

\allowdisplaybreaks[4]

\theoremstyle{plain}
\newtheorem{thm}{Theorem}[section]
\newtheorem{lemm}[thm]{Lemma}
\newtheorem{prop}[thm]{Proposition}

\theoremstyle{definition}

\newtheorem{rem}[thm]{Remark}

\makeatletter

\@addtoreset{equation}{section}
\makeatother

\renewcommand{\div}{\operatorname{div}}
\newcommand{\dB}{\dot{B}}
\newcommand{\fB}{\widehat{\dot{B}}{}}
\newcommand{\dH}{\dot{H}}

\newcommand{\supp}{\operatorname{supp}}

\renewcommand{\leq}{\leqslant}
\renewcommand{\geq}{\geqslant}
\newcommand{\ep}{\varepsilon}

\newcommand{\n}[1]{{\left\|#1\right\|}}

\newcommand{\f}{\frac}
\newcommand{\lp}[1]{\left[#1\right]}
\newcommand{\Mp}[1]{\left\{#1\right\}}
\renewcommand{\sp}[1]{\left(#1\right)}

\newcommand{\rhoe}{\rho_{\varepsilon}}
\newcommand{\ve}{v_{\varepsilon}}
\renewcommand{\ae}{a_{\varepsilon}}

\newcommand{\hve}{\widehat{v_{\varepsilon}}}
\newcommand{\hae}{\widehat{a_{\varepsilon}}}
\newcommand{\hf}{\widehat{f}}
\newcommand{\hg}{\widehat{g}}
\newcommand{\R}{\mathbb{R}}
\newcommand{\Grad}{\nabla}

\newcommand{\Ae}{A_{\varepsilon}}
\newcommand{\Ve}{V_{\varepsilon}}

\newcommand{\ael}{\widetilde{a_{\varepsilon}}}
\newcommand{\vel}{\widetilde{v_{\varepsilon}}}

\begin{document}
\title[Low Mach number limit for Navier--Stokes--Korteweg system]
{Low Mach number limit for the global large solutions to the $2$D Navier--Stokes--Korteweg system in the critical $\widehat{L^p}$ framework}
\author[M.Fujii]{Mikihiro Fujii}
\author[Y.Li]{Yang Li}
\address[M.Fujii]{Graduate School of Science, Nagoya City University, Nagoya, 467-8501, Japan}
\email[M.Fujii]{fujii.mikihiro@nsc.nagoya-cu.ac.jp}
\address[Y.Li]{School of Mathematical Sciences and Center of Pure Mathematics, Anhui University, Hefei, 230601, People's Republic of China}
\email[Y.Li]{lynjum@163.com}
\date{\today} 
\begin{abstract}
    In the present paper, we consider the compressible Navier--Stokes--Korteweg system on the $2$D whole plane and show that a unique global solution exists in the scaling critical Fourier--Besov spaces for arbitrary large initial data provided that the Mach number is sufficiently small.
    Moreover, we also show that the global solution converges to the $2$D incompressible Navier--Stokes flow in the singular limit of zero Mach number. 
    The key ingredient of the proof lies in the nonlinear stability estimates around the large incompressible flow via the Strichartz estimate for the linearized equations in Fourier--Besov spaces. 
\end{abstract}
\maketitle

\tableofcontents

\section{Introduction}\label{sec:intro}

In this paper, we consider the initial value problem of the $2$D compressible Navier--Stokes--Korteweg system, which is provided by Korteweg \cite{Kor-1901} for describing the motion of capillary fluids.
\begin{align}\label{eq:NSK-1}
    \begin{dcases}
        \partial_t \rhoe +\div (\rhoe \ve) = 0, & t>0,x \in \mathbb{R}^2,\\
            \rhoe\sp{\partial_t \ve +  (\ve \cdot \nabla) \ve}
            +\frac{1}{\varepsilon^2}\nabla P(\rhoe)
            = \mathcal{L}\ve + \kappa \rho_{\varepsilon}  \Grad 
 \Delta\rho_{\varepsilon},
        & t>0,x \in \mathbb{R}^2,\\
        \rhoe(0,x)= \rho_{\varepsilon,0}(x),\quad \ve(0,x) = v_0(x), & x \in \mathbb{R}^2,
    \end{dcases}
\end{align}
where 
$\rhoe=\rhoe(t,x):(0,\infty) \times \mathbb{R}^2 \to (0,\infty)$ 
and
$\ve=\ve(t,x):(0,\infty) \times \mathbb{R}^2 \to \mathbb{R}^2$
denote 
the unknown density and velocity field of the fluid, respectively,
while
$\rho_{\varepsilon,0}=\rho_{\varepsilon,0}(x):\mathbb{R}^2 \to (0,\infty)$ 
and
$v_0=v_0(x):\mathbb{R}^2 \to \mathbb{R}^2$ are the given initial data.
The given real analytic function $P=P(\rho)$ of $\rho>0$ represents the pressure of the fluid and satisfy $c_{\infty}:=P'(\rho_{\infty})>0$ for some positive constant $\rho_{\infty}$.
The viscous term $\mathcal{L}\ve$ is given by
\begin{align}
    \mathcal{L}\ve := \mu \Delta \ve + (\mu + \lambda)\nabla \div \ve
\end{align} 
with viscosity coefficients $\mu$ and $\lambda$ satisfying $\mu>0$ and $\nu := 2\mu +\lambda > 0$. 
$\kappa $ is a positive constant acting as the capillary coefficient, while $\ep>0$ is the Mach number defined as the ratio of the speed of flows to that of sound. 
Since the pioneering work of Dunn and Serrin \cite{Dun-Ser-85}, 
there are a large number of mathematical studies on the system \eqref{eq:NSK-1}, see \cites{Ant-Hie-Mar-20,Dan-02-T,Dan-02-R,Dan-05-L,Dan-He-16,Des-Gre-99,Fuj-23-05,Fei-09,Has-20-pro,Hoff,Kre-Lor-Nau-91,Li-Yon-16} and the references therein.  
Especially, in the series of results \cites{Dan-02-T,Dan-02-R,Dan-05-L,Dan-He-16,Fuj-23-05}, they focused on invariant scaling structure and constructed the solutions in the scaling critical Besov spaces framework.
Here, the invariant scaling structure means that for a solution $(\rhoe,\ve)$ to \eqref{eq:NSK-1}, the scaled functions
\begin{align}\label{scaling} 
\rhoe^{(\lambda)}:= \rhoe (\lambda^2 t,\lambda x),
\quad 
\ve^{(\lambda)}:=\lambda \ve (\lambda^2 t,\lambda x)
\end{align}
also solve \eqref{eq:NSK-1} with the pressure replaced by $\lambda^2 P(\rhoe^{(\lambda)})$ for all $\lambda>0$. 
Corresponding to the scaling above, 
the function space $X$ is called the scaling critical if it holds 
$\n{\sp{\rhoe^{(\lambda)},\ve^{(\lambda)}}(0,\cdot)}_{X}
=
\n{\sp{\rhoe,\ve}(0,\cdot)}_{X}$
for all $\lambda>0$.
The aim of this paper is to prove the global existence of a unique solution in the scaling critical Fourier--Besov spaces
$\sp{\fB_{p,1}^{\frac{2}{p}-1}(\mathbb{R}^2) \cap \fB_{p,1}^{\frac{2}{p}}(\mathbb{R}^2)} \times \fB_{p,1}^{\frac{2}{p}-1 }(\mathbb{R}^2) $ with $2\leq p<4$ for arbitrarily large initial disturbance provided that the Mach number $\varepsilon$ is sufficiently small. 
Moreover, we show that the solution converges to the corresponding incompressible flow in the low Mach number limit $\varepsilon \downarrow 0$.

Before we state our main result precisely, we reformulate the system \eqref{eq:NSK-1} and recall some known results related to our study.
We may reformulate the problem in a convenient form.
By considering suitable scaling transform,
we may assume without loss of generality that 
\begin{align}
    \nu=\rho_{\infty}=c_{\infty}=1.
\end{align}
We further assume that the initial density satisfies $\rho_{\varepsilon,0}=1+\varepsilon a_0$ for some bounded function $a_0=a_0(x):\mathbb{R}^2 \to \mathbb{R}$ independent of $\varepsilon$.
Let 
\begin{align}
    \ae:=\frac{\rhoe - 1}{\varepsilon}.
\end{align}
Then, in terms of $(\ae,\ve)$, the system \eqref{eq:NSK-1} is rewritten as 
\begin{align}\label{eq:NSK-2}
    \begin{dcases}
        \partial_t \ae + \dfrac{1}{\varepsilon} \div \ve = -\mathcal{A} \lp{\ae, \ve}, & t>0,x \in \mathbb{R}^2,\\
            \partial_t \ve - \mathcal{L} \ve
            -\kappa \Delta (\varepsilon \nabla \ae)
            +\frac{1}{\varepsilon}\nabla \ae
            = 
            -\mathcal{V}_{\varepsilon}\lp{\ae,\ve},
        & t>0,x \in \mathbb{R}^2,\\
        \ae(0,x)= a_0(x),\quad \ve(0,x) = v_0(x), & x \in \mathbb{R}^2,
    \end{dcases}
\end{align}
where the nonlinear terms $\mathcal{A}\lp{a,v}$ and $\mathcal{V}_{\varepsilon}\lp{a,v}$ are defined as
\begin{align}
    \mathcal{A}\lp{a,v}
    &:=
    \div(av),\\
    \mathcal{V}_{\varepsilon}\lp{a,v}
    &:={}
    (v \cdot \nabla )v 
    +
    \frac{\varepsilon a}{1+\varepsilon a}\mathcal{L}v
    +
    \frac{1}{\varepsilon} 
     \left( \frac{P'(1+\varepsilon a)}{1+\varepsilon a} -1 \right)  \nabla a\\
    &=:
    (v \cdot \nabla )v 
    +\mathcal{J}\sp{\varepsilon a}\mathcal{L}v
    +\frac{1}{\varepsilon} 
    \mathcal{K}\sp{\varepsilon a}  \nabla a.
\end{align}

Let us review the known results related to our study.
Although there are a large number of pioneering works on the well-posedness and incompressible limit for the compressible viscous fluid, we only focus on the results for scaling critical Besov spaces frameworks as this framework is closely related to our study.
We first focus on the usual compressible Navier--Stokes equations, that is equations \eqref{eq:NSK-2} with $\kappa=0$. 
It was Danchin \cites{Dan-01-L,Dan-00-G} who first proved the local and global well-posedness of the compressible Navier--Stokes equations for small data in the scaling critical Besov spaces framework. 
Later, his results are improved by \cites{Che-Mia-Zha-10,Cha-Dan-10,Dan-14,Has-11-1} in terms of the relaxation of smallness condition for the local solutions or general $L^p$-framework.
More precisely, \cites{Che-Mia-Zha-10,Dan-14} showed the existence for local solutions for any $(a_0,v_0) \in (\dB_{p,1}^{\frac{2}{p}-1}(\mathbb{R}^2) \cap \dB_{p,1}^{\frac{2}{p}}(\mathbb{R}^2)) \times \dB_{p,1}^{\frac{2}{p}-1}(\mathbb{R}^2)$ with $1 \leq p < 4$, while \cites{Che-Mia-Zha-15,Iw-Og-22} showed the ill-posedness for $p \geq 4$. In \cites{Cha-Dan-10,Has-11-1}, they proved the global well-posedness for the initial perturbation satisfying
$\n{(a_0,v_0)}_{\dB_{2,1}^{0}}^{\ell;\beta_0} + \n{(\nabla a_0, v_0)}_{\dB_{p,1}^{\frac{2}{p}-1}}^{h;\beta_0} \ll 1$
for $2\leq p < 4$ with some positive constant $\beta_0$.
Here, $\n{\cdot}_{\dB_{2,1}^0}^{\ell;\beta_0}$ and $\n{\cdot}_{\dB_{p,1}^{\frac{2}{p}-1}}^{h;\beta_0}$ denote the truncated Besov semi-norms\footnote{See Section \ref{sec:pre} for the definition.} for low and high frequencies, respectively.
For the incompressible limit, 
Danchin \cite{Dan-02-R} showed the small solution in critical $L^2$-Besov space converges to the incompressible flow as  $\varepsilon \downarrow 0$.
Danchin and He \cite{Dan-He-16} proved the same results in the framework of critical $L^p$ ($2\leq p < 4$) for high frequency part.
Very recently, the first author \cite{Fuj-23-05} relaxed the size condition of the initial data and proved that the global unique solution exists for arbitrary large initial perturbation in the critical $L^2$-Besov space provided that the Mach number is sufficiently small and showed that the solution converges to the corresponding large incompressible flow in some scaling critical norms. 
See \cites{Dan-Muc-17,Dan-Muc-19,Che-Zha-19,Wat-23} for the incompressible limit of $\lambda \to \infty$ type.

Next, we recall the known results for the Navier--Stokes--Korteweg system \eqref{eq:NSK-2} with $\kappa >0$.
Danchin and Desjardins \cite{Dan-Des-01} proved that for small initial perturbation $(a_0,v_0)$ in the scaling critical Besov space $(\dB_{2,1}^{0}(\mathbb{R}^2) \cap \dB_{2,1}^{1}(\mathbb{R}^2)) \times \dB_{2,1}^0(\mathbb{R}^2)$, there exists a global solution to \eqref{eq:NSK-2} with $\varepsilon=1$ in the class $a,\nabla a, v \in C([0,\infty);\dB_{2,1}^0(\mathbb{R}^2)) \cap L^1(0,\infty;\dB_{2,1}^2(\mathbb{R}^2))$.
See \cite{Kaw-Shi-Xu-21} for the results on the high frequency $L^p$ framework corresponding to \cites{Cha-Dan-10,Has-11-1}.
Here, we remark that the $L^1$-maximal regularity 
$a,\nabla a \in L^1(0,\infty;\dB_{2,1}^2(\mathbb{R}^2))$ for the density perturbation comes from the dissipation due to the Korteweg term $\kappa \Delta \nabla a$.
This additional dissipative structure provides us better regularity results than that for usual compressible Navier--Stokes equations. 
Indeed, \cite{Cha-Dan-Xu-21} proved the global small solutions to the compressible Navier--Stokes--Korteweg system possesses the analytic Gevrey regularity.
Kobayashi and Nakasato \cite{Kob-Nak-pre} made use of the Fourier--Besov spaces\footnote{See Section \ref{sec:pre} for the definition.} to improve the $L^2$-regularity for the low frequency part in \cites{Kaw-Shi-Xu-21} and show the global well-posedness of momentum formulated Navier--Stokes--Korteweg system in $(\fB_{p,1}^{\frac{2}{p}-1}(\mathbb{R}^2) \cap \fB_{p,1}^{\frac{2}{p}}(\mathbb{R}^2)) \times \fB_{p,1}^{\frac{2}{p}-1}(\mathbb{R}^2)$ with $1 \leq p \leq \infty$.

In the present paper, we prove the existence of a unique global solution to \eqref{eq:NSK-1} for arbitrarily \emph{large} initial data in the scaling critical Fourier--Besov space $(\fB_{p,1}^{\frac{2}{p}-1}(\mathbb{R}^2) \cap \fB_{p,1}^{\frac{2}{p}}(\mathbb{R}^2) ) \times \fB_{p,1}^{\frac{2}{p}-1}(\mathbb{R}^2)$ if the Mach number $\ep$ is suitably small. Furthermore, we consider the singular limit of zero Mach number and show that the global solution of the compressible Navier--Stokes--Korteweg system converges to the incompressible Navier--Stokes flow as $\varepsilon \downarrow 0$ in {critical} Fourier--Besov spaces. 

Our main result of this paper reads as follows:
\begin{thm}\label{thm}
    Let $2 \leq p <  4$.
    Then, for any $a_0 \in \fB_{p,1}^{\frac{2}{p}-1}(\mathbb{R}^2) \cap \fB_{p,1}^{\frac{2}{p}}(\mathbb{R}^2)$ and $v_0 \in \fB_{p,1}^{\frac{2}{p}-1}(\mathbb{R}^2)$, there exists a positive constant $\varepsilon_0=\varepsilon_0(p,\mu,\kappa,P,a_0,v_0)$ such that for any $0<\varepsilon \leq \varepsilon_0$, the system \eqref{eq:NSK-2} possesses a unique global solution $(\ae,\ve)$ in the class 
    \begin{align}
        \ae, \varepsilon \nabla \ae, \ve 
        \in \widetilde{C}([0,\infty);\fB_{p,1}^{\frac{2}{p}-1}(\mathbb{R}^2)) \cap L^1(0,\infty;\fB_{p,1}^{\frac{2}{p}+1}(\mathbb{R}^2))
    \end{align}
    with $\rhoe(t,x) = 1 +\varepsilon \ae(t,x) >0$.
    Moreover, for $p<q\leq \infty$ and $1<r<\infty$, it holds
    \begin{align}
        \lim_{\varepsilon \to +0}
        \sp{
    \n{(\ae,\mathbb{Q}\ve)}_{\widetilde{L^r}(0,\infty;\dB_{q,1}^{\frac{2}{q}-1+\frac{2}{r}})}
        +
        \n{\mathbb{P}\ve - w}_{\widetilde{L^{\infty}}(0,\infty;\fB_{p,1}^{\frac{2}{p}-1})\cap L^1(0,\infty;\fB_{p,1}^{\frac{2}{p}+1})}
        }
        =0.
    \end{align}
    Here, $\mathbb{P}:=I+\nabla \div (-\Delta)^{-1}$ denotes the Helmholtz projection, $\mathbb{Q}:=I-\mathbb{P}$ is its orthogonal projection, and $w \in \widetilde{C}([0,\infty);\fB_{p,1}^{\frac{2}{p}-1}(\mathbb{R}^2)) \cap L^1(0,\infty;\fB_{p,1}^{\frac{2}{p}+1}(\mathbb{R}^2))$ is the unique global solution to the $2$D incompressible Navier--Stokes equations
    \begin{align}\label{eq:lim-1}
        \begin{cases}
            \partial_t w - \mu \Delta w + \mathbb{P} (w \cdot \nabla)w = 0, & t>0,x \in \mathbb{R}^2,\\
            \div w = 0, & t \geq 0, x \in \mathbb{R}^2,\\
            w(0,x) = \mathbb{P}v_0(x), & x \in \mathbb{R}^2.
        \end{cases}
    \end{align}
\end{thm}

\begin{rem}
Let us compare Theorem \ref{thm} to related studies.
In the first author's previous study \cite{Fuj-23-05}, Theorem \ref{thm} with $p=2$ and $\kappa = 0$ is proved, that is, for any large data in the critical $L^2$ Besov space, there exists a unique global solution to the original compressible Navier--Stokes system \eqref{eq:NSK-2} with $\kappa=0$ provided that $\varepsilon$ is sufficiently small. 
He also proved that the solution converges to the corresponding incompressible flow as $\varepsilon \downarrow 0$ in some scaling critical norms. 
However, the proof of \cite{Fuj-23-05} will fails if we consider the critical $L^p$ framework, as in \cite{Dan-He-16}, due to some technical reasons related to high frequency regularity of the density perturbation. 
On the other hand, considering the compressible Naiver--Stokes--Korteweg system, the Korteweg term provides us a better regularity for the density perturbation, which enables us to overcome the aforementioned difficulty.
Moreover, our proof relies on the Gronwall lemma, which is much more simpler than that for the previous method \cite{Fuj-23-05}.
\end{rem}

We explain some notations which are used throughout this paper. 
Various generic positive constants are denoted by the same $C$ changing from line to line. For any $1\leq p \leq \infty$, we denote by $p'$ the H\"{o}lder conjugate of $p$. For two Banach spaces $X_1,X_2$ with $X_1\cap X_2 \neq \varnothing$, we set $\n{\cdot}_{ X_1\cap X_2 }:=\n{\cdot}_{X_1}+\n{\cdot}_{X_2}$. 
For any $z\in \mathbb{C}$, we denote by $\Re z$ its real part.

The rest of this paper is arranged as follows. 
In Section \ref{sec:pre}, we introduce the function spaces and prepare some useful lemmas. 
In Section \ref{sec:lin}, we perform a detailed analysis for the linearized equations. 
The proof of our main Theorem \ref{thm} is then completed in Section \ref{sec:pf} via a stability argument. 
Global well-posedness to the $2$D incompressible Navier--Stokes equations in the critical Fourier--Besov spaces is given in Appendix \ref{sec:a} and the local well-posedness to the compressible Navier--Stokes--Korteweg system is given in Appendix \ref{sec:LWP}.

\section{Preliminaries}\label{sec:pre}
\subsection{Function spaces}
In this subsection, we introduce the function spaces which will be used throughout this paper. We refer to \cites{Bah-Che-Dan-11,Chi-18,Sa-18} for more details.  
Let $\mathscr{S}(\R^2)$ be the set of all Schwartz functions on $\R^2$ and $\mathscr{S}'(\R^2)$ be the set of all tempered distributions on $\R^2$. For any $f\in \mathscr{S}(\R^2) $, we define the Fourier transform and the inverse Fourier transform as
\begin{align}
\mathscr{F}[f](\xi) =\hf(\xi)
=\int_{\R^2} e^{-i x \cdot \xi} f(x) dx,
\quad  
\mathscr{F}^{-1}[f](x) =(2\pi)^{-2} 
\int_{\R^2} e^{i x \cdot \xi} f(\xi) d\xi.
\end{align}
For any $1\leq p \leq \infty$, we denote by $L^p(\R^2)$ the usual Lebesgue space and $\widehat{L^p}(\R^2)$ the Fourier--Lebesgue space defined by
\begin{align}
\widehat{L^p}(\R^2):=
\left\{
f\in \mathscr{S}'(\R^2); \hf \in L^{1}_{\rm loc}(\R^2),\,\, 
\| f \|_{\widehat{L^p}}:=\| \hf \|_{L^{p'}} <\infty
\right\}. 
\end{align}
A sequence of functions $\{\phi_j \}_{j\in \mathbb{Z}} \subset \mathscr{S}(\R^2)$ is called the Littlewood--Paley decomposition if  
\begin{align}
 & \phi_j (x)=2^{2j}\phi_0(2^j x),
 \quad 0\leq \widehat{\phi_0}    \leq 1, \\
 & \rm{supp} \,\widehat{\phi_0}  \subset
 \left\{ 
 \xi \in \R^2; 2^{-1}\leq |\xi| \leq 2
 \right\}, \\
 & \sum_{j\in \mathbb{Z}} 
 \widehat{\phi_j}(\xi)=1,\,\, \text{ for all  }
 \xi \in \R^2 \backslash \{0\}. 
\end{align}
For any $j\in \mathbb{Z}$ and $f \in \mathscr{S}'(\R^2)$, we define the dyadic operator and low frequency cut-off operator as 
\begin{align}
\Delta_j f:= \mathscr{F}^{-1}[  \widehat{\phi_j}  \hf],
\qquad S_j f := \sum_{j'\leq j} \Delta_{j'}f. 
\end{align}
For any $s\in \R,1\leq p,\sigma \leq \infty$, we define the homogeneous Besov space $\dot{B}^s_{p,\sigma}(\R^2)$ as 
\begin{align}
\dot{B}^s_{p,\sigma} (\R^2)  & :=
\left\{
f \in \mathscr{S}'(\R^2)/ \mathscr{P}(\R^2); 
\| f\|_{\dot{B}^s_{p,\sigma}} <\infty
\right\}, \\
\| f\|_{\dot{B}^s_{p,\sigma}}& := 
\Big\|   
\{    2^{js} \| \Delta_j f \|_{L^p}  \}_{j\in \mathbb{Z}}
\Big \|_{\ell^{\sigma}}, 
\end{align}
where $\mathscr{P}(\R^2)$ denotes the set of all polynomials on $\R^2$. It is classical that when $s<2/p$ or $(s,\sigma)=(2/p,1)$, $\dot{B}^s_{p,\sigma}(\R^2)$ may be identified as 
\begin{align}
\left\{
f \in \mathscr{S}'(\R^2); 
\| f\|_{\dot{B}^s_{p,\sigma}} <\infty,\,\,
f=\sum_{j\in \mathbb{Z}} \Delta_j f \,\text{  in  } \mathscr{S}'(\R^2)
\right\}. 
\end{align}
For any $s\in \R,1\leq p,\sigma \leq \infty$, we define the homogeneous Fourier--Besov space $ \fB^{s}_{p,\sigma}(\R^2)$ as 
\begin{align}
\fB^s_{p,\sigma} (\R^2)  & :=
\left\{
f \in \mathscr{S}'(\R^2)/ \mathscr{P}(\R^2); 
\| f\|_{\fB^s_{p,\sigma}} <\infty
\right\}, \\
\| f\|_{\fB^s_{p,\sigma}}& := 
\Big\|   
\{    2^{js} 
\| \Delta_j f \|
_{\widehat{L^p}}  \}_{j\in \mathbb{Z}}
\Big \|_{\ell^{\sigma}}
=  \Big\|   
\{    2^{js} 
\| \widehat{\phi_j} \hf \|
_{L^{p'}}   \}_{j\in \mathbb{Z}}
\Big \|_{\ell^{\sigma}}. 
\end{align}  
For any $s\in \R,1\leq p_1\leq p_2 \leq \infty,1\leq \sigma_1\leq \sigma_2 \leq \infty$, it holds the continuous embedding of Sobolev type:
\begin{align} 
\fB^s_{p_1,\sigma_1} (\R^2) 
\hookrightarrow
\fB^{s-2(\f{1}{p_1}-\f{1}{p_2})}_{p_2,\sigma_2} (\R^2) .
\end{align}
Moreover, for any $s\in \R, 2\leq p_1\leq p_2 \leq \infty,1\leq \sigma_1\leq \sigma_2 \leq \infty$, the Hausdorff--Young inequality implies
\begin{align}
\fB^s_{p_1,\sigma_1} (\R^2) 
\hookrightarrow
\dot{B}^{s-2(\f{1}{p_1}-\f{1}{p_2})}_{p_2,\sigma_2} (\R^2) . 
\end{align} 
In case of $p=2$, the Plancherel theorem ensures that $\fB^s_{2,\sigma} (\R^2)=\dot{B}^{s}_{2,\sigma}(\R^2)$ in the sense of norm equivalence. Taking the time variable into account, it is natural to consider the Chemin--Lerner type spaces inspired by Chemin and Lerner \cite{Che-Ler-95}. For any $s\in \R,1\leq p,\sigma,r \leq \infty$ and an interval $I \subset \R$, we define the Chemin--Lerner space $\widetilde{L^r}(I;\dot{B}^{s}_{p,\sigma}(\R^2))$ as
\begin{align}
\widetilde{L^r}(I;\dot{B}^{s}_{p,\sigma}(\R^2))  & :=
\left\{
f :  I \rightarrow \mathscr{S}'(\R^2)/ \mathscr{P}(\R^2); 
\| f\|_{\widetilde{L^r}(I;\dot{B}^{s}_{p,\sigma})} <\infty
\right\}, \\
\| f\|_{\widetilde{L^r}(I;\dot{B}^{s}_{p,\sigma})} & := 
\Big\|   
\{    2^{js} \| \Delta_j f \|_{ L^r(I; L^p)    }  \}_{j\in \mathbb{Z}}
\Big \|_{\ell^{\sigma}}.  
\end{align}
Analogously, in the context of Fourier--Besov space, we define the Chemin--Lerner space $\widetilde{L^r}(I;\fB^{s}_{p,\sigma}(\R^2))$ as
\begin{align}
\widetilde{L^r}(I;\fB^{s}_{p,\sigma}(\R^2))  & :=
\left\{
f :  I \rightarrow \mathscr{S}'(\R^2)/ \mathscr{P}(\R^2); 
\| f\|_{\widetilde{L^r}(I;\fB^{s}_{p,\sigma})} <\infty
\right\}, \\
\| f\|_{\widetilde{L^r}(I;\fB^{s}_{p,\sigma})} & := 
\Big\|   
\{    2^{js} \| \Delta_j f \|_{ L^r(I; \widehat{L^p} )    }  \}_{j\in \mathbb{Z}}
\Big \|_{\ell^{\sigma}}
=\Big\|   
\{    2^{js} \| \widehat{\phi_j} \hf  \|_{ L^r(I; L^{p'} )    }  \}_{j\in \mathbb{Z}}
\Big \|_{\ell^{\sigma}}. 
\end{align}
For brevity, we set
\begin{align}
 \widetilde{C}([0,\infty);\fB^s_{p,\sigma}(\R^2))
 :=
 C([0,\infty);\fB^s_{p,\sigma}(\R^2))
 \cap
 \widetilde{L^{\infty}}
 (0,\infty;\fB^s_{p,\sigma}(\R^2)).
\end{align}
As a direct consequence of Minkowski inequality, we have the following embeddings
\begin{align}
 & \widetilde{L^r}(I;\dot{B}^{s}_{p,\sigma}(\R^2))
   \hookrightarrow
   L^r(I;\dot{B}^{s}_{p,\sigma}(\R^2))  
   \quad  \text{  if  } r \geq \sigma,  \\
  &  L^r(I;\dot{B}^{s}_{p,\sigma}(\R^2)) 
   \hookrightarrow 
   \widetilde{L^r}(I;\dot{B}^{s}_{p,\sigma}(\R^2))
   \quad  \text{  if  } r \leq \sigma,   \\
   & \widetilde{L^r}(I;\fB^{s}_{p,\sigma}(\R^2))
   \hookrightarrow
   L^r(I;\fB^{s}_{p,\sigma}(\R^2)) 
     \quad  \text{  if  } r \geq \sigma,  \\
     &  L^r(I;\fB^{s}_{p,\sigma}(\R^2)) 
     \hookrightarrow
     \widetilde{L^r}(I;\fB^{s}_{p,\sigma}(\R^2))
      \quad  \text{  if  } r \leq \sigma. 
\end{align}
For any $0<\alpha<\beta$, we define the truncated semi-norms according to high, middle and low frequency parts as 
\begin{align}
 \| f\|_{\fB^s_{p,\sigma}}^{h;\beta}  &:=
 \Big\|   
\{    2^{js} 
\| \Delta_j f \|
_{\widehat{L^p}}  \}_{2^{j} \geq \beta}
\Big \|_{\ell^{\sigma}},  \\
\| f\|_{\fB^s_{p,\sigma}}^{m;\alpha,\beta}  &:=
 \Big\|   
\{    2^{js} 
\| \Delta_j f \|
_{\widehat{L^p}}  \}_{ \alpha \leq  2^{j} < \beta}
\Big \|_{\ell^{\sigma}},  \\ 
\| f\|_{\fB^s_{p,\sigma}}^{\ell;\alpha}  &:=
 \Big\|   
\{    2^{js} 
\| \Delta_j f \|
_{\widehat{L^p}}  \}_{2^{j} < \alpha}
\Big \|_{\ell^{\sigma}},  \\
\| f\|_{\widetilde{L^r}(I;\fB^{s}_{p,\sigma})}^{h;\beta} & := 
\Big\|   
\{    2^{js} \| \Delta_j f \|_{ L^r(I; \widehat{L^p} )    }  \}_{2^{j} \geq \beta}
\Big \|_{\ell^{\sigma}}, \\ 
\| f\|_{\widetilde{L^r}(I;\fB^{s}_{p,\sigma})}^{m;\alpha,\beta} & := 
\Big\|   
\{    2^{js} \| \Delta_j f \|_{ L^r(I; \widehat{L^p} )    }  \}_{ \alpha \leq 2^{j} < \beta}
\Big \|_{\ell^{\sigma}}, \\ 
\| f\|_{\widetilde{L^r}(I;\fB^{s}_{p,\sigma})}^{\ell;\alpha} & := 
\Big\|   
\{    2^{js} \| \Delta_j f \|_{ L^r(I; \widehat{L^p} )    }  \}_{2^{j} < \alpha}
\Big \|_{\ell^{\sigma}}. 
\end{align}

\subsection{Useful lemmas}
In this subsection, we prepare several useful lemmas which play a crucial role in the proof of our main Theorem \ref{thm}.

We begin with the maximal regularity estimates for the heat equation in Chemin--Lerner spaces of Fourier--Besov type. For the proof, we refer to Nakasato \cite{Nak-22}. 
\begin{lemm}\label{lemm:max-heat}
Let $1\leq p,\sigma \leq \infty,\mu>0,s\in \R,0<T\leq \infty$ and $1\leq r_1\leq r \leq \infty$. Consider the heat equation 
\begin{align}
        \begin{cases}
            \partial_t u - \mu \Delta u = f, & t>0,x \in \mathbb{R}^2,\\
            u(0,x) = u_0(x), & x \in \mathbb{R}^2
        \end{cases}
\end{align}
with $u_0\in \fB^s_{p,\sigma}(\R^2)$ and $f\in \widetilde{L^{r_1}}(0,T;\fB^{s-2+\frac{2}{r_1} }_{p,\sigma}(\R^2))$. Then, there exists a positive constant $C=C(r_1,r,\mu)$ such that 
\begin{align}
\| u \|_{ 
\widetilde{L^{r}}(0,T;\fB^{s+\frac{2}{r} }_{p,\sigma})    
}
\leq 
C \left( 
\|   u_0 \|_{ 
\fB^s_{p,\sigma}  
}
+
\|  f \|_{ 
\widetilde{L^{r_1}}(0,T;\fB^{s-2+\frac{2}{r_1} }_{p,\sigma} )
}
\right).
\end{align}
\end{lemm}

To consider the product estimate of two functions in Besov spaces, we recall the Bony para-product decomposition as follows:
\begin{align}
fg=T_f g + T_g f +R(f,g),
\end{align}
where
\begin{align}
 T_f g & :=\sum_{k\in \mathbb{Z} }S_{k-3} f \Delta_{k} g = 
 \sum_{k\in \mathbb{Z} }
 \left(
 \sum_{j \leq k-3} \Delta_{j} f
 \right) 
 \Delta_{k} g, \\
 R( f ,g) & :=\sum_{| k-j|\leq 2 }
 \Delta_{k} f \Delta_{j} g.
\end{align}

Next, we recall the following lemma on the operators $T,R$ in Besov spaces and Fourier--Besov spaces. 
\begin{lemm}\label{lemm-Tfg-Rfg}
\begin{enumerate}[(1)]
    \item { Let $1\leq p,p_1,p_2,\sigma \leq \infty$ and $s,s_1,s_2\in \R$ be subject to 
    \begin{align}
 \f{1}{p}=\f{1}{p_1}+\f{1}{p_2},\quad
  s=s_1+s_2, \quad
 s_1\leq 0.
    \end{align}
    Then, there exists a positive constant $C=C(s_1,s_2)$ such that
    \begin{align}
     \n{ T _f g  }_{ \dot{B}^{s}_{p,\sigma} }
\leq C
  \n{ f   }_{ \dot{B}^{s_1}_{p_1,1} }
  \n{ g }_{ \dot{B}^{s_2}_{p_2,\sigma} }
     \end{align}
     for all $f\in \dot{B}^{s_1}_{p_1,1}(\R^2)$ and $g \in \dot{B}^{s_2}_{p_2,\sigma}(\R^2)$, and 
     \begin{align}
     \n{ T_fg }_{ \fB^{s}_{p,\sigma} } 
\leq C
  \n{ f   }_{ \fB^{s_1}_{p_1,1} }
  \n{ g }_{ \fB^{s_2}_{p_2,\sigma} }
     \end{align}
     for all $f\in \fB^{s_1}_{p_1,1}(\R^2)$ and $g \in \fB^{s_2}_{p_2,\sigma}(\R^2)$.  
    }
    \item{ Let $1\leq p,p_1,p_2,\sigma,\sigma_1,\sigma_2 \leq \infty$ and $s,s_1,s_2\in \R$ be subject to 
    \begin{align}
 \f{1}{p}=\f{1}{p_1}+\f{1}{p_2},\quad
  s=s_1+s_2>0, \quad
 \f{1}{\sigma} \leq \f{1}{\sigma_1}+\f{1}{\sigma_2}      .
    \end{align}
     Then, there exists a positive constant $C=C(s_1,s_2)$ such that
     \begin{align}
     \n{ R (f ,g)  }_{ \dot{B}^{s}_{p,\sigma} }
\leq C
  \n{ f   }_{ \dot{B}^{s_1}_{p_1,\sigma_1} }
  \n{ g }_{ \dot{B}^{s_2}_{p_2,\sigma_2} }
     \end{align}
     for all $f\in \dot{B}^{s_1}_{p_1,\sigma_1}(\R^2)$ and $g \in \dot{B}^{s_2}_{p_2,\sigma_2}(\R^2)$, and 
     \begin{align}
     \n{ R (f ,g)  }_{ \fB^{s}_{p,\sigma} }
\leq C 
  \n{ f   }_{ \fB^{s_1}_{p_1,\sigma_1} }
  \n{ g }_{ \fB^{s_2}_{p_2,\sigma_2} }
     \end{align}
     for all $f\in \fB^{s_1}_{p_1,\sigma_1}(\R^2)$ and $g \in \fB^{s_2}_{p_2,\sigma_2}(\R^2)$. 
    }
\end{enumerate}
  
\end{lemm}
For the proof of Lemma \ref{lemm-Tfg-Rfg}, we refer to Bahouri et al. \cite{Bah-Che-Dan-11} for the Besov case, and Nakasato \cite{Nak-22} for the Fourier--Besov case.

Lemma \ref{lemm-Tfg-Rfg} immediately yields the following product estimate.
\begin{lemm}\label{lemm-fg-general}
Let $1\leq p,p_1,p_2,\sigma \leq \infty,\alpha_1,\alpha_2 \geq 0, s\in \R$ be subject to
\begin{align}
  s+\min\Mp{\f{2}{p_1},\f{2}{p'}} > 0.
\end{align}
Then, there exists a positive constant $C=C(p,p_1,p_2,\sigma,\alpha_1,\alpha_2)$ such that
\begin{align}
 &
 \n{fg}_{ \dot{B}^s_{p,\sigma}  }
 \leq 
 C
 \left(
\n{f}_{ \dot{B}^{ \f{2}{p_1} -\alpha_1 }_{p_1,1}  }
\n{g}_{ \dot{B}^{ s +\alpha_1 }_{p,\sigma}  }
+
\n{f}_{ \dot{B}^{ s +\alpha_2 }_{p,\sigma}  }
\n{g}_{ \dot{B}^{ \f{2}{p_2} -\alpha_2 }_{p_2,1}  }
 \right)
\end{align}
for all $f \in  \dB_{p_1,1}^{\frac{2}{p_1}-\alpha_1}(\mathbb{R}^2) \cap \dB_{p,\sigma}^{s+\alpha_2}(\mathbb{R}^2)$ and $g \in \dB_{p,\sigma}^{s+\alpha_1}(\mathbb{R}^2) \cap \dB_{p_2,1}^{\frac{2}{p_2}-\alpha_2}(\mathbb{R}^2)$, and 
\begin{align}
 \n{fg}_{ \fB^s_{p,\sigma}  }
 \leq 
 C
 \left(
\n{f}_{ \fB^{ \f{2}{p_1} -\alpha_1 }_{p_1,1}  }
\n{g}_{ \fB^{ s +\alpha_1 }_{p,\sigma}  }
+
\n{f}_{ \fB^{ s +\alpha_2 }_{p,\sigma}  }
\n{g}_{ \fB^{ \f{2}{p_2} -\alpha_2 }_{p_2,1}  }
 \right)
\end{align}
for all $f \in  \fB_{p_1,1}^{\frac{2}{p_1}-\alpha_1}(\mathbb{R}^2) \cap \fB_{p,\sigma}^{s+\alpha_2}(\mathbb{R}^2)$ and $g \in \fB_{p,\sigma}^{s+\alpha_1}(\mathbb{R}^2) \cap \fB_{p_2,1}^{\frac{2}{p_2}-\alpha_2}(\mathbb{R}^2)$.
\end{lemm}
\begin{rem}
    From the second estimate above and  $\fB_{\infty,1}^0(\mathbb{R}^2)=\widehat{L^{\infty}}(\mathbb{R}^2)$, it follows that 
    for 
    $s>0$, $1\leqslant p,q \leqslant \infty$, 
    there exists a positive constant $C=C(p,q,s)$ such that
    \begin{align}\label{prod-est-B}
		\| fg \|_{\fB_{p,q}^s}
		\leqslant
		C \left ( \|f\|_{\fB_{p,q}^s}\|g\|_{\widehat{L^{\infty}}} 
	    +
        \|f\|_{\widehat{L^{\infty}}}\|g\|_{\fB_{p,q}^s}
		\right).
    \end{align}
    Here, we should remark that the estimate \eqref{prod-est-B} holds for the end-point case $(s,p,q)=(0,\infty,1)$ since $\n{fg}_{\widehat{L^{\infty}}} \leq \n{f}_{\widehat{L^{\infty}}}\n{g}_{\widehat{L^{\infty}}}$.
\end{rem}

In the next lemma, we recall the para-product estimate for truncated Besov norms. For the proof, we refer to \cite{Fuj-23-05}.
\begin{lemm}\label{lemm-Tfg-low}
Let $1\leq p,p_1,p_2,\sigma\leq \infty$ and $s,s_1,s_2\in \R$ be subject to
\begin{align}
 \f{1}{p}=\f{1}{p_1}+\f{1}{p_2},\quad
 s=s_1+s_2, \quad
 s_1\leq 0.
\end{align}
Then, there exists a positive constant $C=C(s_1,s_2)$ such that 
\begin{align}
&
\n{ T _f g  }_{ \dot{B}^{s}_{p,\sigma} }^{\ell;\beta}
\leq 
C
\n{ f }_{ \dot{B}^{s_1}_{p_1,1} }^{\ell;\beta}
\n{ g }_{ \dot{B}^{s_2}_{p_2,\sigma} }^{\ell;4 \beta},\\
&
\n{ T _f g  }_{ \dot{B}^{s}_{p,\sigma} }^{h;\beta}
\leq 
C
 \n{ f }_{ \dot{B}^{s_1}_{p_1,1} }
\n{ g }_{ \dot{B}^{s_2}_{p_2,\sigma} }^{h; \f{\beta}{4} }
\end{align}
for all $\beta>0$ and all $f,g$ such that the right-hand side are finite. 
\end{lemm}

Let us end this section by preparing the composition lemma in the \emph{Fourier--Besov} spaces.
\begin{lemm}\label{lemm-comp} 
Let $1 \leq p \leq \infty$ and $s \in \mathbb{R}$ satisfy 
\begin{align}
    -\min 
    \Mp{\frac{2}{p},\frac{2}{p'}}
    <s\leq 
    \frac{2}{p}.
\end{align}
Let $F=F(s)$ be a $C^1$ function defined on some open interval including $0$.
Assume that $F(0)=0$ and $F$ is real analytic at $s=0$; $R_0$ denotes its radius of convergence.
Then, there exist two constants $0<\delta_0<1$ and $C=C(p,s,F,R_0)>0$ such that 
\begin{align}
    \n{F(u)}_{\fB_{p,1}^s}
    \leq
    C
    \n{u}_{\fB_{p,1}^s},
    \qquad
    \n{F(u)-F(v)}_{\fB_{p,1}^s}
    \leq
    C
    \n{u-v}_{\fB_{p,1}^s},
\end{align}
for all $u,v \in \fB_{p,1}^s(\mathbb{R}^2) \cap \fB_{p,1}^{\frac{2}{p}}(\mathbb{R}^2)$ satisfying $\n{u}_{\fB_{p,1}^{\frac{2}{p}}}, \n{v}_{\fB_{p,1}^{\frac{2}{p}}} \leq \delta_0  R_0$. 
\end{lemm}
\begin{proof} 
    The strategy of the proof is based on \cite{Kob-Nak-pre}.
    We first notice that by Lemma \ref{lemm-fg-general}, there exists a constant $C_{*} > 1$ such that 
    \begin{align}
        \n{fg}_{\fB_{p,1}^{\frac{2}{p}}}
        \leq
        C_*
        \n{f}_{\fB_{p,1}^{\frac{2}{p}}}
        \n{g}_{\fB_{p,1}^{\frac{2}{p}}}
    \end{align}
    for all $f,g \in \fB_{p,1}^{\frac{2}{p}}(\mathbb{R}^2)$.
    Let $\delta_0:=1/C_*^2$.
    We expand $F$ as
    \begin{align}
        F(s)=\sum_{n=1}^{\infty}a_n s^n, \qquad |s| < R_0.
    \end{align}
    Then, it holds 
    \begin{align}
        \n{F(u)}_{\fB_{p,1}^s}
        \leq{}&
        \sum_{n=1}^{\infty}
        |a_n|
        \n{u^n}_{\fB_{p,1}^s}\\
        \leq{}&
        C
        \sum_{n=1}^{\infty}
        |a_n|
        \n{u^{n-1}}_{\fB_{p,1}^{\frac{2}{p}}}
        \n{u}_{\fB_{p,1}^s}\\
        \leq{}&
        C
        \sum_{n=1}^{\infty}
        |a_n|
        \sp{C_*\n{u}_{\fB_{p,1}^{\frac{2}{p}}}}^{n-1}
        \n{u}_{\fB_{p,1}^s}\\
        \leq{}&
        C
        \left(
        \sum_{n=1}^{\infty}
        |a_n|
        \sp{ \sqrt{\delta_0} R_0}^{n-1}
        \right) 
        \n{u}_{\fB_{p,1}^s},
    \end{align}
    which provides the first estimate.

    For the second estimate, the mean value formula implies
    \begin{align}
        F(u)-F(v)
        ={}&
        \int_0^1
        F'(\theta u + (1-\theta)v)d\theta 
        (u-v)
        \\
        ={}&
        \int_0^1
        \sum_{n=1}^{\infty}na_n(\theta u + (1-\theta)v)^{n-1}d\theta 
        (u-v)
    \end{align}
    and
    we have
    \begin{align}
        \n{F'(\theta u + (1-\theta)v)}_{\fB_{p,1}^{\frac{2}{p}}}
        &\leq
        \sum_{n=1}^{\infty}
        n|a_n|
        \n{\sp{\theta u + (1-\theta)v}^{n-1}}_{\fB_{p,1}^{\frac{2}{p}}}\\
        &\leq
        C
        \sum_{n=1}^{\infty}
        n|a_n|
        \sp{C_*\n{\theta u + (1-\theta)v}_{\fB_{p,1}^{\frac{2}{p}}}}^{n-1}\\
        &\leq
        C
        \sum_{n=1}^{\infty}
        n|a_n|\sp{ \sqrt{\delta_0} R_0}^{n-1}.
    \end{align} 
    Thus, we obtain
    \begin{align}
        \n{F(u)-F(v)}_{\fB_{p,1}^s}
        &\leq
        C
        \int_0^1
        \n{F'(\theta u + (1-\theta)v)}_{\fB_{p,1}^{\frac{2}{p}}}
        d\theta \times 
        \n{u-v}_{\fB_{p,1}^s}
        \\
        &\leq
        C
        \left( 
        \sum_{n=1}^{\infty}
        n|a_n|\sp{  \sqrt{\delta_0} R_0}^{n-1} \right)
        \n{u-v}_{\fB_{p,1}^s},
    \end{align}
    which completes the proof.
\end{proof}

\section{Linear estimates}\label{sec:lin}
In this section, we establish several estimates for the solutions to the linearized system:
\begin{align}\label{eq:NSK-lin-1}
    \begin{dcases}
        \partial_t \ae + \dfrac{1}{\varepsilon} \div \ve = f, & t>0,x \in \mathbb{R}^2,\\
            \partial_t \ve - \mathcal{L} \ve
            -\kappa \Delta (\varepsilon \nabla \ae)
            +\frac{1}{\varepsilon}\nabla \ae
            = 
            g,
        & t>0,x \in \mathbb{R}^2,\\
        \ae(0,x)= a_0(x),\quad \ve(0,x) = v_0(x), & x \in \mathbb{R}^2,
    \end{dcases}
\end{align}  
where $f$ and $g$ are given external forces. To begin with, we prove the maximal regularity estimates for the solutions to \eqref{eq:NSK-lin-1} in Chemin--Lerner spaces. 
\begin{lemm}\label{lemm:lin-ene}
    There exists a positive constant $C$ such that
    \begin{align}
        \n{(\ae,\varepsilon \nabla \ae, \ve)}_{\widetilde{L^r}(0,T;\fB_{p,\sigma}^{s+\frac{2}{r}})}
        \leq{}&
        C\n{(a_0,\varepsilon \nabla a_0, v_0)}_{\fB_{p,\sigma}^s}\\
        &+
        C
        \n{(f,\varepsilon \nabla f, g)}_{\widetilde{L^1}(0,T;\fB_{p,\sigma}^s)}
    \end{align}
    for all $1 \leq p,\sigma,  r \leq \infty$, $s \in \mathbb{R}$, and $0<T \leq \infty$.
\end{lemm}
\begin{proof}
Let us first establish the pointwise estimate in Fourier space. Applying the Fourier transform of \eqref{eq:NSK-lin-1}, we see that 
    \begin{align}\label{eq:lin-hat}
        \begin{dcases}
        \partial_t \hae + \dfrac{1}{\varepsilon} i\xi \cdot \hve = \hf,\\
        \partial_t \hve + \mu |\xi|^2 \hve + (\mu+\lambda)\xi (\xi \cdot \hve)
        +\kappa |\xi|^2 (\varepsilon i\xi  \hae)
        +\frac{1}{\varepsilon}i\xi \hae
        = 
        \hg,\\
        \hae(0,\xi)= \widehat{a_0}(\xi),\quad \hve(0,\xi) = \widehat{v_0}(\xi).
    \end{dcases}
    \end{align}
    Multiply the first equation of \eqref{eq:lin-hat} by $\hae$ and take $\mathbb{C}^2$-inner product of the second equation of \eqref{eq:lin-hat}with $\hve$.
    Then, summing up them and taking real part of it, we have
    \begin{align}\label{lin-ene-1-1}
        \frac{1}{2}
        \partial_t
        | (\hae,\hve) |^2
        +
        \underline{\mu}
        |\xi|^2
        | \hve |^2 
        +
        \kappa 
        |\xi|^2 
        \Re 
        \langle \varepsilon i\xi  \hae, \hve \rangle_{\mathbb{C}^2}
        \leq
        | (\hf, \hg) | 
        | (\hae,\hve) |,
    \end{align}
    where $\underline{\mu}:=\min \{ \mu,1 \}$ and we have used 
    \begin{align}
        &
        \langle \mu |\xi|^2 \hve + (\mu+\lambda)\xi (\xi \cdot \hve), \hve \rangle_{\mathbb{C}^2}
        \geq
        \underline{\mu}
        |\xi|^2
        | \hve |^2,\\
        &
        \Re(i\xi \cdot \hve \hae 
        + 
        \langle i\xi \hae , \hve\rangle_{\mathbb{C}^2})
        =
        0.
    \end{align}
    Multiplying the first equation of \eqref{eq:lin-hat} by $\varepsilon i\xi$ and taking the $\mathbb{C}^2$-inner product of it with $\varepsilon i \xi \hae$, and then taking the real part, we see that
    \begin{align}\label{lin-ene-1-2}
        \frac{1}{2}
        \partial_t
        |\varepsilon i \xi \hae|^2 
        -
        |\xi|^2
        \Re 
        \langle \varepsilon i\xi  \hae, \hve \rangle_{\mathbb{C}^2}
        =
        |\varepsilon i \xi \hf|
        |\varepsilon i \xi \hae|.
    \end{align}
    Next, multiplying the first equation of \eqref{eq:lin-hat} by $i\varepsilon\xi$ and then taking $\mathbb{C}^2$-inner product of it with $\hve$, we have
    \begin{align}\label{lin-ene-1-au}
        \langle i\varepsilon\xi \partial_t\hae , \hve \rangle_{\mathbb{C}^2}
        -
        |\xi \cdot \hve|^2
        =
        \langle i\varepsilon\xi \widehat{f}, \hve \rangle_{\mathbb{C}^2}.
    \end{align}
    By the inner product of the second equation of \eqref{eq:lin-hat} with $i\varepsilon \xi \hae$, it holds
    \begin{align}\label{lin-ene-1-ua}
        \langle i\varepsilon \xi \hae, \partial_t\hve \rangle_{\mathbb{C}^2}
        +
        |\xi|^2
        \langle i\varepsilon \xi \hae, \hve \rangle_{\mathbb{C}^2}
        +
        \kappa |\xi|^2 |\varepsilon i \xi \hae |^2
        +
        |\xi|^2|\hae|^2
        =
        \langle i\varepsilon \xi \hae, \widehat{g} \rangle_{\mathbb{C}^2}.
    \end{align}
    Summing up \eqref{lin-ene-1-au} and \eqref{lin-ene-1-ua},
    and taking the real part of it,
    we have 
    \begin{align}\label{lin-ene-1-3}
        \begin{split}
        &
        \partial_t
        \Re 
        \langle i\varepsilon \xi \hae, \hve \rangle_{\mathbb{C}^2}
        +
        \kappa |\xi|^2 |\varepsilon i \xi \hae |^2
        +
        |\xi|^2|\hae|^2\\
        &\quad\leq{}
        |\varepsilon i \xi \widehat{f}|
        |\hve|
        +
        |\varepsilon i\xi \hae|
        |\hg|
        +
        |\xi|^2
        |\varepsilon i \xi \hae|
        |\hve|
        +
        |\xi|^2|\hve|^2.
        \end{split}
    \end{align}
    Here, we define
    \begin{align}
        \widehat{V}^2(t,\xi)
        :=
        |(\hae,\hve)|^2
        +
        \kappa
        |\varepsilon i \xi \hae |^2
        +
        2\delta
        \Re 
        \langle i\varepsilon \xi \hae, \hve \rangle_{\mathbb{C}^2},
    \end{align}
    where we have set 
    \begin{align}
        \delta:=\min\Mp{
        \frac{1}{2},
        \frac{\kappa}{2}
        }
    \end{align}
    As there holds
    \begin{align}
        \left|
        \widehat{V}^2(t,\xi)
        -
        \sp{|(\hae,\hve)|^2
        +
        \kappa
        |\varepsilon i \xi \hae |^2}
        \right|
        \leq{}&
        2\delta |i\varepsilon \xi \hae| |\hve|\\
        \leq{}&
        \delta |i\varepsilon \xi \hae|^2 
        +
        \delta |\hve|^2\\
        \leq{}&
        \frac{1}{2}
        \sp{|(\hae,\hve)|^2
        +
        \kappa
        |\varepsilon i \xi \hae |^2},
    \end{align}
    we see that 
    \begin{align}\label{est-V}
        \frac{1}{2}
        \sp{|(\hae,\hve)|^2
        +
        \kappa
        |\varepsilon i \xi \hae |^2}
        \leq
        \widehat{V}^2(t,\xi)
        \leq
        \frac{3}{2}
        \sp{|(\hae,\hve)|^2
        +
        \kappa
        |\varepsilon i \xi \hae |^2}.
    \end{align}
    By \eqref{lin-ene-1-1}, \eqref{lin-ene-1-2} and \eqref{lin-ene-1-3}, we have
    \begin{align}
        &
        \frac{1}{2}
        \partial_t 
        \widehat{V}^2
        +
        \underline{\mu}
        |\xi|^2
        |\hve|^2
        +
        \delta
        |\xi|^2
        |\hae|^2
        +
        \delta
        \kappa
        |\xi|^2
        |\varepsilon i \xi \hae|^2\\
        &
        \quad
        \leq 
        (|(\hf,\hg)|+|\varepsilon i\xi \hf|)
        (|(\hae,\hve)|^2
        +
        \kappa
        |\varepsilon i \xi \hae |^2)
        +
        \delta
        |\xi|^2
        |\varepsilon i \xi \hae|
        |\hve|
        +
        \delta
        |\xi|^2|\hve|^2.
    \end{align}
    From this with \eqref{est-V}, we obtain 
    \begin{align}
        &
        \frac{1}{2}
        \partial_t 
        \widehat{V}^2
        +
        \frac{\delta}{2}
        |\xi|^2
        \widehat{V}^2
        \leq 
        C(|(\hf,\hg)|+|\varepsilon i\xi \hf|)
        \widehat{V},
    \end{align}
    which implies 
    \begin{align}\label{est-V-2}
        \widehat{V}(t,\xi)
        \leq{}&
        Ce^{-\frac{\delta}{2}|\xi|^2t}
        \widehat{V}(0,\xi)\\
        &
        +
        C
        \int_0^t
        e^{-\frac{\delta}{2}|\xi|^2(t-\tau)}
        (|(\hf,\hg)(\tau,\xi)|+|\varepsilon i\xi \hf(\tau,\xi)|)
        d\tau.
    \end{align}
    Combining \eqref{est-V} and \eqref{est-V-2}, we obtain 
    \begin{align}
        &
        |(\hae,\hve)(t,\xi)|
        +
        \kappa
        |\varepsilon i \xi \hae(t,\xi) |\\
        &\quad 
        \leq
        Ce^{-\frac{\delta}{2}|\xi|^2t}
        \sp{|(\widehat{a_0},\widehat{v_0})(\xi)|
        +
        \kappa
        |\varepsilon i \xi \widehat{a_0}(\xi) |
        }\\
        &\qquad
        +
        C
        \int_0^t
        e^{-\frac{\delta}{2}|\xi|^2(t-\tau)}
        (|(\hf,\hg)(\tau,\xi)|+|\varepsilon i \xi \hf(\tau,\xi)|)
        d\tau.
    \end{align}
    Multiplying this by $\widehat{\phi_j}(\xi)$ and taking $L^r(0,T;L^{p'}(\mathbb{R}^2_{\xi}))$-norm, we complete the proof. 
\end{proof}

Next, we focus on the asymptotic behavior of solutions to \eqref{eq:NSK-lin-1} as $\varepsilon \downarrow 0$.
\begin{lemm}\label{lemm:lin-0-1}
    Let $1 \leq p \leq \infty$.
    Let $(\ae,\ve)$ be the solution to \eqref{eq:NSK-lin-1} with the data $a_0 \in \fB_{p,1}^{\frac{2}{p}-1}(\mathbb{R}^2) \cap \fB_{p,1}^{\frac{2}{p}}(\mathbb{R}^2)$, $v_0 \in \fB_{p,1}^{\frac{2}{p}-1}(\mathbb{R}^2)$, $f \in L^1(0,\infty;\fB_{p,1}^{\frac{2}{p}-1}(\mathbb{R}^2) \cap \fB_{p,1}^{\frac{2}{p}}(\mathbb{R}^2))$ and $g \in L^1(0,\infty;\fB_{p,1}^{\frac{2}{p}-1}(\mathbb{R}^2))$.
    Then, it holds 
    \begin{align}
        \lim_{\varepsilon \downarrow 0} 
        \n{\varepsilon \nabla \ae }_{\widetilde{L^{\infty}}(0,\infty;\fB_{p,1}^{\frac{2}{p}-1})}
        =0.
    \end{align}
\end{lemm}
\begin{proof}
    Let $\delta$ be an arbitrary positive constant.
    Then, since
    \begin{align}
        &\lim_{L\to \infty}
        \sum_{2^j \geq L}
        2^{(\frac{2}{p}-1)j}
        \n{\Delta_j \sp{a_0,\nabla a_0,v_0}}_{\widehat{L^p}}
        =0,\\
        &
        \lim_{L\to \infty}
        \sum_{2^j \geq L}
        2^{(\frac{2}{p}-1)j}
        \n{\Delta_j \sp{f,\nabla f,g}}_{L^1(0,\infty;\widehat{L^p})}
        =0, 
    \end{align}
    there exists a constant $L_{\delta}=L_{\delta}(p, a_0,v_0,f,g )>1$ such that 
    \begin{align}
        \n{\sp{a_0,\nabla a_0,v_0}}_{\fB_{p,1}^{\frac{2}{p}-1}}^{h;L_{\delta}} \leq \delta,\quad
        \n{\sp{f,\nabla f,g}}_{L^1(0,\infty;\fB_{p,1}^{\frac{2}{p}-1})}^{h;L_{\delta}} \leq \delta.
    \end{align}
    Let $0<\varepsilon<1/L_{\delta}$.
    Then, we infer from Lemma \ref{lemm:lin-ene} and the Bernstein inequality that 
    \begin{align}
        &\n{\varepsilon \nabla \ae }_{\widetilde{L^{\infty}}(0,\infty;\fB_{p,1}^{\frac{2}{p}-1})}\\
        &\quad={}
        \n{\varepsilon \nabla \ae }_{\widetilde{L^{\infty}}(0,\infty;\fB_{p,1}^{\frac{2}{p}-1})}^{  \ell;L_{\delta}  }
        +
        \n{\varepsilon \nabla \ae }_{\widetilde{L^{\infty}}(0,\infty;\fB_{p,1}^{\frac{2}{p}-1})}^{m;L_{\delta},\frac{1}{\varepsilon}}
        +
        \n{\varepsilon \nabla \ae }_{\widetilde{L^{\infty}}(0,\infty;\fB_{p,1}^{\frac{2}{p}-1})}^{h;\frac{1}{\varepsilon}}\\
        &\quad
        \leq{}
        \varepsilon L_{\delta}
        \n{ \ae }_{\widetilde{L^{\infty}}(0,\infty;\fB_{p,1}^{\frac{2}{p}-1})}^{  
  \ell; L_{\delta}   } 
        +
        \n{ \ae }_{\widetilde{L^{\infty}}(0,\infty;\fB_{p,1}^{\frac{2}{p}-1})}^{m;L_{\delta},\frac{1}{\varepsilon}}
        +
        \n{\varepsilon \nabla \ae }_{\widetilde{L^{\infty}}(0,\infty;\fB_{p,1}^{\frac{2}{p}-1})}^{h;\frac{1}{\varepsilon}}\\
        &\quad
        \leq{}
        \varepsilon L_{\delta}
        \n{ \sp{\ae,\varepsilon\nabla\ae,\ve} }_{\widetilde{L^{\infty}}(0,\infty;\fB_{p,1}^{\frac{2}{p}-1})}
        +
        \n{ \sp{\ae,\varepsilon\nabla\ae,\ve} }_{\widetilde{L^{\infty}}(0,\infty;\fB_{p,1}^{\frac{2}{p}-1})}^{h;L_{\delta}}\\
        &\quad
        \leq{}
        C\varepsilon L_{\delta}
        \sp{
        \n{(a_0,\varepsilon \nabla a_0, v_0)}_{\fB_{p,1}^{{\frac{2}{p}-1}}}
        +
        \n{(f,\varepsilon \nabla f, g)}_{\widetilde{L^1}(0,T;\dB_{p,1}^{\frac{2}{p}-1})}}\\
        &\qquad
        +
        C
        \sp{
        \n{(a_0,\varepsilon \nabla a_0, v_0)}_{\fB_{p,1}^{{\frac{2}{p}-1}}}^{h;L_{\delta}}
        +
        \n{(f,\varepsilon \nabla f, g)}_{\widetilde{L^1}(0,T;\dB_{p,1}^{\frac{2}{p}-1})}^{h;L_{\delta}}}\\
        &\quad
        \leq{}
        C\varepsilon L_{\delta}
        \sp{
        \n{(a_0,\varepsilon \nabla a_0, v_0)}_{\fB_{p,1}^{{\frac{2}{p}-1}}}
        +
        \n{(f,\varepsilon \nabla f, g)}_{\widetilde{L^1}(0,T;\dB_{p,1}^{\frac{2}{p}-1})}}
        +
        C\delta.
    \end{align}
    Thus, we have
    \begin{align}
        \limsup_{\varepsilon \downarrow 0}
        \n{\varepsilon \nabla \ae }_{\widetilde{L^{\infty}}(0,\infty;\fB_{p,1}^{\frac{2}{p}-1})}
        \leq
        C\delta.
    \end{align}
    Since $\delta>0$ is arbitrary, we complete the proof.
\end{proof}
\begin{lemm}\label{lemm:lin-0-2}
    Let $2\leq p <q \leq \infty$ and $1<r<\infty$.
    Let $(\ae,\ve)$ be the solution to \eqref{eq:NSK-lin-1} with the data $a_0 \in \fB_{p,1}^{\frac{2}{p}-1}(\mathbb{R}^2) \cap \fB_{p,1}^{\frac{2}{p}}(\mathbb{R}^2)$, $v_0 \in \fB_{p,1}^{\frac{2}{p}-1}(\mathbb{R}^2)$, $f \in L^1(0,\infty;\fB_{p,1}^{\frac{2}{p}-1}(\mathbb{R}^2) \cap \fB_{p,1}^{\frac{2}{p}}(\mathbb{R}^2))$ and $g \in L^1(0,\infty;\fB_{p,1}^{\frac{2}{p}-1}(\mathbb{R}^2))$.
    Then, it holds
    \begin{align}
        \lim_{\varepsilon \downarrow0}
        \n{\sp{\ae,\varepsilon\nabla\ae,\mathbb{Q}\ve}}_{\widetilde{L^r}(0,\infty;\dB_{q,1}^{\frac{2}{q}-1+\frac{2}{r}})}
        =0.
    \end{align}
\end{lemm}
In order to prove Lemma \ref{lemm:lin-0-2}, we prepare the following lemma about the Strichartz estimates for the solutions to \eqref{eq:NSK-lin-1}.
\begin{lemm}\label{lemm:str-hat-Besov-1}
    Let $p$, $q$, and $r$ satisfy 
    \begin{align}
        2 \leq p \leq q \leq \infty, \qquad 
        0 
        \leq 
        \frac{1}{r}
        \leq 
        \frac{1}{2}
        \sp{
        \frac{1}{p} - \frac{1}{q}
        }.
    \end{align}
    Then, there exists a positive constant $C=C(p,q,r)$ such that
    the solution $(\ae,\ve)$ to \eqref{eq:NSK-lin-1} satisfies 
    \begin{align}
        \n{(\ae,\mathbb{Q}\ve)}_{\widetilde{L^r}(0,T;\dB_{q,\sigma}^{\frac{2}{q}+\frac{1}{r}+s})}
        \leq{}&
        C
        \varepsilon^{\frac{1}{r}}
        \n{(a_0, \varepsilon \nabla a_0 ,\mathbb{Q}v_0)}_{\fB_{p,\sigma}^{\frac{2}{p}+s}}\\
        &+
        C
        \varepsilon^{\frac{1}{r}}
        \n{(f,\varepsilon \nabla f,\mathbb{Q}g)}_{\widetilde{L^1}(0,T;\fB_{p,\sigma}^{\frac{2}{p}+s})}
    \end{align}
    for all $s\in \mathbb{R}$, $\varepsilon > 0$, $1 \leq \sigma \leq \infty$, $0 < T \leq \infty$.
\end{lemm}
\begin{proof}
    We first show the following Strichartz estimates for the wave equations with initial data in the Fourier--Besov spaces: 
    \begin{align}\label{st-f}
        \n{e^{i\frac{t}{\varepsilon}|\nabla|}f}_{\widetilde{L^r}(\mathbb{R};\dB_{q,\sigma}^{\frac{2}{q}+\frac{1}{r}+s})}
        \leq
        C
        \varepsilon^{\frac{1}{r}}
        \n{f}_{\fB_{p,\sigma}^{\frac{2}{p}+s}}
    \end{align}
    for all $f \in \fB_{p,\sigma}^{\frac{2}{p}+s}(\mathbb{R}^2)$.
    Let $0\leq \theta \leq 1$ and $r_1 $ satisfy  
    \begin{align}
        \frac{1}{p} = \frac{\theta}{2} + \frac{1-\theta}{q},\qquad
        \frac{1}{r} = \frac{\theta}{r_1} + \frac{1-\theta}{\infty}.
    \end{align}
    Then, we see that
    \begin{align}
        2 \leq q \leq \infty,\qquad
        0 
        \leq
        \frac{1}{r_1}
        \leq 
        \frac{1}{2}
        \sp{
        \frac{1}{2} - \frac{1}{q}
        }.
    \end{align}
    Thus, by the standard Strichartz estimate for the $2$D wave propagator, there holds 
    \begin{align}
        \n{e^{it|\nabla|}\Delta_0 f}_{L^{r_1}(\mathbb{R};L^{q})}
        \leq
        C
        \n{f}_{L^2}
    \end{align}
    for all $f \in L^2(\mathbb{R}^2)$; see \cite{Bah-Che-Dan-11}*{Theorem 8.18} for the proof.
    On the other hand, we see that 
    \begin{align}
        \n{e^{it|\nabla|}\Delta_0 f}_{L^{\infty}(\mathbb{R};L^{q})}
        \leq
        \n{e^{it|\xi|}\widehat{\phi_0}(\xi) \widehat{f}(\xi)}_{L^{\infty}(\mathbb{R};L^{q'})}
        \leq
        \n{\widehat{f}(\xi)}_{L^{q'}}
        =
        \n{f}_{\widehat{L^{q}}}
    \end{align}
    for all $f \in \widehat{L^{q}}(\mathbb{R}^2)$.
    Hence, from the complex interpolation, 
    we obtain 
    \begin{align}
        \n{e^{it|\nabla|}\Delta_0 f}_{L^r(\mathbb{R};L^q)}
        \leq
        C
        \n{f}_{\widehat{L^p}}
    \end{align}
    for all $f \in \widehat{L^p}(\mathbb{R}^2)$.
    Let $\widetilde{\Delta_j}:=\Delta_{j-1} + \Delta_j +\Delta_{j+1}$.
    Then, we see that 
    \begin{align}
        e^{i\frac{t}{\varepsilon}|\nabla|}\Delta_j f(x)
        ={}&
        e^{i\frac{t}{\varepsilon}|\nabla|}\Delta_j \sp{\widetilde{\Delta_j}f}(x)\\
        ={}&
        e^{i\frac{2^jt}{\varepsilon}|\nabla|}\Delta_0
        \lp{\sp{\widetilde{\Delta_j}f}\sp{2^{-j}\cdot}}(2^jx).
    \end{align}
    Thus, we have 
    \begin{align}
        \n{e^{i\frac{t}{\varepsilon}|\nabla|}\Delta_j f}_{L^r(\mathbb{R};L^q)}
        ={}&
        \varepsilon^{\frac{1}{r}}
        2^{-(\frac{2}{q}+\frac{1}{r})j}
        \n{e^{it|\nabla|}\Delta_0 \lp{\sp{\widetilde{\Delta_j}f}\sp{2^{-j}\cdot}}}_{L^r(\mathbb{R};L^q)}\\
        \leq{}&
        C\varepsilon^{\frac{1}{r}}
        2^{-(\frac{2}{q}+\frac{1}{r})j}
        \n{{\widetilde{\Delta_j}f}\sp{2^{-j}\cdot}}_{\widehat{L^p}}\\
        ={}&
        C\varepsilon^{\frac{1}{r}}
        2^{(\frac{2}{p}-\frac{2}{q}-\frac{1}{r})j}
        \n{{\widetilde{\Delta_j}f}}_{\widehat{L^p}}\\
        ={}&
        C\varepsilon^{\frac{1}{r}}
        2^{(\frac{2}{p}-\frac{2}{q}-\frac{1}{r})j}
        \n{{\widetilde{\Delta_j}(\Delta_jf)}}_{\widehat{L^p}}\\
        \leq{}&
        C\varepsilon^{\frac{1}{r}}
        2^{(\frac{2}{p}-\frac{2}{q}-\frac{1}{r})j}
        \n{{{\Delta_j}f}}_{\widehat{L^p}},
    \end{align}
    which completes the proof of \eqref{st-f} via multiplying both sides by $2^{ (\f{2}{q}+\f{1}{r}+s)j  }$, followed by taking $\ell^{\sigma}(\mathbb{Z})$ norm over $j\in \mathbb{Z}$.

Let $\widetilde{v_\varepsilon}:= |\nabla|^{-1}\div \mathbb{Q} \ve$.
Then, direct calculation shows that $(\ae,\vel)$ satisfies 
\begin{align}
    \begin{dcases}
    \partial_t \ae + \frac{1}{\varepsilon}|\nabla| \widetilde{v_\varepsilon} = f,\\
    \partial_t \widetilde{v_\varepsilon} - \frac{1}{\varepsilon}|\nabla| \ae = \Delta \widetilde{v_\varepsilon} + |\nabla|^{-1}\div \mathbb{Q}(g+\kappa \Delta (\varepsilon \nabla \ae)),\\
    \ae(0,x)=a_0(x),\quad \widetilde{v_\varepsilon}(0,x)=\widetilde{v_0}(x):=|\nabla|^{-1}\div \mathbb{Q}v_0(x).
    \end{dcases}
\end{align}
Thus, we see that
\begin{align}
    \begin{pmatrix}
    \ae(t)\\
    \widetilde{v_\varepsilon}(t)
    \end{pmatrix}
    ={}&
    U\left(\frac{t}{\varepsilon}\right)
    \begin{pmatrix}
    a_0\\
    \widetilde{v_0}
    \end{pmatrix}\\
    &
    +
    \int_{0}^t
    U\left(\frac{t-\tau}{\varepsilon}\right)
    \begin{pmatrix}
    f\\
    \Delta \widetilde{v_\varepsilon} + |\nabla|^{-1}\div \mathbb{Q}(g+\kappa \Delta (\varepsilon \nabla \ae))
    \end{pmatrix}
    (\tau)d\tau,
\end{align}
where 
\begin{align}
    U\left( t \right)
    :={}&
    \begin{pmatrix}
    \cos |\nabla|t & -\sin |\nabla| t \\
    \sin |\nabla|t & \cos|\nabla|t
    \end{pmatrix}. 
\end{align} 
Then, from the Strichartz estimate \eqref{st-f} we see
\begin{align}\label{wave-1}
    \| (\ae, \widetilde{\ve}) \|_{\widetilde{L^r}(0,T;\dB_{q,\sigma}^{\frac{2}{q}+\frac{1}{r}+s})}
    \leqslant{}&
    C\varepsilon^{\frac{1}{r}}
    \bigg(
    \| (a_0,\mathbb{Q}v_0)\|_{\fB_{p,\sigma}^{\frac{2}{p}+s}}
    +
    \| (f,\mathbb{Q}g) \|_{ \widetilde{L^1}(0,T;\fB_{p,\sigma}^{\frac{2}{p}+s})}
    \bigg)\\
    &
    +
    C\varepsilon^{\frac{1}{r}}
    \| \mathbb{Q}\ve \|_{\widetilde{L^1}(0,T;\fB_{p,\sigma}^{\frac{2}{p}+2+s})}
    +
    C\varepsilon^{\frac{1}{r}}
    \| \varepsilon \nabla \ae \|_{ \widetilde{L^1} (0,T;\fB_{p,\sigma}^{\frac{2}{p}+2+s})}.
\end{align}
From Lemma \ref{lemm:lin-ene}, we have 
\begin{align}\label{wave-2}
    &
    \| \mathbb{Q}{\ve} \|_{\widetilde{L^1}(0,T;\fB_{p,\sigma}^{\frac{2}{p}+2+s})}
    +
    \| \varepsilon \nabla \ae \|_{\widetilde{L^1}(0,T;\fB_{p,\sigma}^{\frac{2}{p}+2+s})}\\
    &\quad
    \leqslant
    C
    \| (a_0, \varepsilon \nabla a_0, \mathbb{Q}v_0) \|_{\fB_{p,\sigma}^{\frac{2}{p}+s}}
    +
    C
    \| (f, \varepsilon \nabla f, \mathbb{Q}g) \|_{\widetilde{L^1}(0,T;\fB_{p,\sigma}^{\frac{2}{p}+s})}.
\end{align}
Combining \eqref{wave-1}, \eqref{wave-2} and noticing that
\begin{align}
    \| \mathbb{Q}\ve \|_{\widetilde{L^r}(0,T;\dB_{q,\sigma}^{\frac{2}{q}+\frac{1}{r}+s})}
    =
    \| \nabla|\nabla|^{-1}\widetilde{\ve} \|_{\widetilde{L^r}(0,T;\dB_{q,\sigma}^{\frac{2}{q}+\frac{1}{r}+s})}
    \leqslant
    C 
    \| \widetilde{\ve} \|_{\widetilde{L^r}(0,T;\dB_{q,\sigma}^{\frac{2}{q}+\frac{1}{r}+s})},
\end{align}
we complete the proof.
\end{proof}
Using Lemma \ref{lemm:str-hat-Besov-1}, we prove Lemma \ref{lemm:lin-0-2}. 
\begin{proof}[Proof of Lemma \ref{lemm:lin-0-2}]
We first consider the case 
\begin{align}
        0 
        <
        \frac{1}{r}
        \leq 
        \frac{1}{2}
        \sp{
        \frac{1}{p} - \frac{1}{q}
        }.
    \end{align}
Let $\delta > 0$. 
Then, since
\begin{align}
    &\lim_{L\to \infty}
    \sum_{2^j \geq L}
    2^{(\frac{2}{p}-1)j}
    \n{\Delta_j \sp{a_0,\nabla a_0,v_0}}_{\widehat{L^p}}
    =0,\\
    &
    \lim_{L\to \infty}
    \sum_{2^j \geq L}
    2^{(\frac{2}{p}-1)j}
    \n{\Delta_j \sp{f,\nabla f,g}}_{L^1(0,\infty;\widehat{L^p})}
    =0,
\end{align}
there exists a constant $L_{\delta}=L_{\delta}(p,a_0,v_0,f,g)>1$ such that 
\begin{align}
    \n{\sp{a_0,\nabla a_0,v_0}}_{\fB_{p,1}^{\frac{2}{p}-1}}^{h;L_{\delta}} \leq \delta,\quad
    \n{\sp{f,\nabla f,g}}_{L^1(0,\infty;\fB_{p,1}^{\frac{2}{p}-1})}^{h;L_{\delta}} \leq \delta.
\end{align}
Then by Lemma \ref{lemm:lin-ene}, we have 
\begin{align}
    &\n{\sp{\ae,\varepsilon \nabla \ae,\ve}}_{\widetilde{L^{\infty}}(0,\infty;\fB_{p,1}^{\frac{2}{p}-1}) \cap L^1(0,\infty;\fB_{p,1}^{\frac{2}{p}+1})}^{h;L_{\delta}}\\
    &\quad
    \leqslant
    C
    \sp{\n{\sp{a_0,\nabla a_0,v_0}}_{\fB_{p,1}^{\frac{2}{p}-1}}^{h;L_{\delta}} 
    +
    \n{\sp{f,\nabla f,g}}_{L^1(0,\infty;\fB_{p,1}^{\frac{2}{p}-1})}^{h;L_{\delta}}}
    \\
    &\quad
    \leqslant
    C 
    \delta. 
\end{align}
Let $0<\varepsilon<1/L_{\delta}$.
We infer from Lemma \ref{lemm:str-hat-Besov-1}, Bernstein inequality and the continuous embedding $\fB_{p,1}^{\frac{2}{p}-1 +\f{2}{r} }(\R^2) \hookrightarrow \dB_{q,1}^{\frac{2}{q}-1+\frac{2}{r}}(\R^2)  $ that 
\begin{align}
    &
    \n{\sp{\ae,\varepsilon \nabla \ae,\mathbb{Q}\ve}}_{\widetilde{L^r}(0,\infty;\dB_{q,1}^{\frac{2}{q}-1+\frac{2}{r}})}\\
    &\quad
    \leqslant
    2
    \n{\sp{\ae,\mathbb{Q}\ve}}_{\widetilde{L^r}(0,\infty;\dB_{q,1}^{\frac{2}{q}-1+\frac{2}{r}})}^{\ell;\frac{1}{\varepsilon}}
    +
    \n{\sp{\ae,\varepsilon \nabla \ae,\ve}}_{\widetilde{L^r}(0,\infty;\dB_{q,1}^{\frac{2}{q}-1+\frac{2}{r}})}^{h;\frac{1}{\varepsilon}}\\
    &\quad
    \leqslant
    C
    \n{\sp{\ae,\mathbb{Q}\ve}}_{\widetilde{L^r}(0,\infty;\dB_{q,1}^{\frac{2}{q}-1+\frac{2}{r}})}^{\ell;L_{\delta}}\\
    &\qquad
    +
    C
    \n{\sp{\ae,\ve}}_{\widetilde{L^r}(0,\infty;\fB_{p,1}^{\frac{2}{p}-1+\frac{2}{r}})}^{m;L_{\delta},\frac{1}{\varepsilon}}
    +
    \n{\sp{\ae,\varepsilon \nabla \ae,\ve}}_{\widetilde{L^r}(0,\infty;\fB_{p,1}^{\frac{2}{p}-1+\frac{2}{r}})}^{h;\frac{1}{\varepsilon}}\\
    &\quad
    \leqslant
    C
    L_{\delta}^{\frac{1}{r}}
    \n{\sp{\ae,\mathbb{Q}\ve}}_{\widetilde{L^r}(0,\infty;\dB_{q,1}^{\frac{2}{q}-1+\frac{1}{r}})}^{\ell;L_{\delta}}
    +
    C
    \n{\sp{\ae,\varepsilon \nabla \ae,\ve}}_{\widetilde{L^{\infty}}(0,\infty;\fB_{p,1}^{\frac{2}{p}-1}) \cap L^1(0,\infty;\fB_{p,1}^{\frac{2}{p}+1})}^{h;L_{\delta} }\\
    &\quad
    \leqslant
    C
    L_{\delta}^{\frac{1}{r}}
    \varepsilon^{\frac{1}{r}}
    \sp{
    \n{\sp{a_0,\varepsilon \nabla a_0,\mathbb{Q}v_0}}_{\fB_{p,1}^{\frac{2}{p}-1}}
    +
    \n{\sp{f,\varepsilon\nabla f,g}}_{L^1(0,\infty;\fB_{p,1}^{\frac{2}{p}-1})}
    }+C\delta,
\end{align} 
which implies 
\begin{align}
    \limsup_{\varepsilon \downarrow 0}
    \n{\sp{\ae,\varepsilon \nabla \ae,\mathbb{Q}\ve}}_{\widetilde{L^r}(0,\infty;\dB_{q,1}^{\frac{2}{q}-1+\frac{2}{r}})}
    \leq
    C\delta.
\end{align}
Since $\delta>0$ is arbitrary, we have $\n{\sp{\ae,\varepsilon \nabla \ae,\mathbb{Q} \ve}}_{\widetilde{L^r}(0,\infty;\dB_{q,1}^{\frac{2}{q}-1+\frac{2}{r}})} \to 0$ as $\varepsilon \downarrow 0$.

Next, we consider the case 
\begin{align}
    \frac{1}{2}
    \sp{
    \frac{1}{p} - \frac{1}{q}
    }<\frac{1}{r}<1.
\end{align}
We choose a $r<r_1<\infty$ such that 
\begin{align}
    0<\frac{1}{r_1}\leq
    \frac{1}{2}
    \sp{
    \frac{1}{p} - \frac{1}{q}
    }
\end{align}
and let $0<\theta<1$ satisfy
\begin{align}
    \frac{1}{r}=\frac{\theta}{r_1} + \frac{1-\theta}{1}.
\end{align}
Then, by the interpolation inequality, Lemma \ref{lemm:lin-ene}, and the proved conclusion that $\n{\sp{\ae,\varepsilon \nabla \ae,\mathbb{Q}\ve}}_{\widetilde{L^{r_1}}(0,\infty;\dB_{q,1}^{\frac{2}{q}-1+\frac{2}{r_1}})} \to 0$ as $\varepsilon \downarrow 0$, we have
\begin{align}
    &
    \n{\sp{\ae,\varepsilon \nabla \ae,\mathbb{Q}\ve}}_{\widetilde{L^r}(0,\infty;\dB_{q,1}^{\frac{2}{q}-1+\frac{2}{r}})}\\
    &
    \quad
    \leq
    \n{\sp{\ae,\varepsilon \nabla \ae,\mathbb{Q}\ve}}_{\widetilde{L^{r_1}}(0,\infty;\dB_{q,1}^{\frac{2}{q}-1+\frac{2}{r_1}})}^{\theta}
    \n{\sp{\ae,\varepsilon \nabla \ae,\ve}}_{L^1(0,\infty;\dB_{q,1}^{\frac{2}{q}+1})}^{1-\theta}\\
    &
    \quad
    \leq
    \n{\sp{\ae,\varepsilon \nabla \ae,\mathbb{Q}\ve}}_{\widetilde{L^{r_1}}(0,\infty;\dB_{q,1}^{\frac{2}{q}-1+\frac{2}{r_1}})}^{\theta}
    \n{\sp{\ae,\varepsilon \nabla \ae,\ve}}_{L^1(0,\infty;\fB_{p,1}^{\frac{2}{p}+1})}^{1-\theta}\\
    &
    \quad
    \leq
    C
    \n{\sp{\ae,\varepsilon \nabla \ae,\mathbb{Q}\ve}}_{\widetilde{L^{r_1}}(0,\infty;\dB_{q,1}^{\frac{2}{q}-1+\frac{2}{r_1}})}^{\theta}\\
    &\qquad
    \times
    \sp{
    \n{\sp{a_0,\varepsilon \nabla a_0,v_0}}_{\fB_{p,1}^{\frac{2}{p}-1}}
    +
    \n{\sp{f,\varepsilon\nabla f,g}}_{L^1(0,\infty;\fB_{p,1}^{\frac{2}{p}-1})}
    }^{1-\theta}\\
    &\quad \to 0 \qquad {\rm as}\quad \varepsilon \downarrow 0.
\end{align}
This completes the proof.
\end{proof}

\section{Proof of Theorem \ref{thm}}\label{sec:pf}
The aim of this section is to prove our main Theorem \ref{thm}.
We first provide the following proposition for the $2$D incompressible Navier--Stokes equations
\begin{align}\label{eq:lim-2}
        \begin{cases}
            \partial_t w - \mu \Delta w + \mathbb{P} (w \cdot \nabla)w = 0, & t>0,x \in \mathbb{R}^2,\\
            \div w = 0, & t \geq 0, x \in \mathbb{R}^2,\\
            w(0,x) = w_0(x), & x \in \mathbb{R}^2
        \end{cases}
\end{align}
in {\it the critical Fourier--Besov spaces} $\fB_{p,1}^{\frac{2}{p}-1}(\mathbb{R}^2)$.
More precisely, we prove the following proposition.
\begin{prop}\label{prop:lim}
    Let $2\leq p \leq 4$.
    Then, for any divergence free vector field $w_0 \in \fB_{p,1}^{\frac{2}{p}-1}(\mathbb{R}^2)$, the equations \eqref{eq:lim-2} possess a global solution 
    \begin{align}
        w \in \widetilde{C}([0,\infty);\fB_{p,1}^{\frac{2}{p}-1}(\mathbb{R}^2)) \cap L^1(0,\infty;\fB_{p,1}^{\frac{2}{p}+1}(\mathbb{R}^2)).
    \end{align}
\end{prop}

The second proposition is concerned with the \emph{local} well-posedness for the compressible Navier--Stokes--Korteweg system. 
\begin{prop}\label{lemm:LWP-CNSK}
    Let $1 \leq p \leq \infty$.
    Then, for any $a_0 \in \fB_{p,1}^{\frac{2}{p}-1}(\mathbb{R}^2) \cap \fB_{p,1}^{\frac{2}{p}}(\mathbb{R}^2)$ and $v_0 \in \fB_{p,1}^{\frac{2}{p}-1}(\mathbb{R}^2)$, there exists a positive constant $\varepsilon_0=\varepsilon_0(\mu,\kappa,p,a_0,v_0)$ and a positive time $T_0=T_0(\mu,\kappa,p,a_0,v_0)$ such that for any $0<\varepsilon \leq \varepsilon_0$, the system \eqref{eq:NSK-2} possesses a unique local solution $(\ae,\ve)$ in the class 
    \begin{align}
        \ae, \varepsilon \nabla \ae, \ve 
        \in \widetilde{C}([0,T_0];\fB_{p,1}^{\frac{2}{p}-1}(\mathbb{R}^2)) \cap L^1(0,T_0;\fB_{p,1}^{\frac{2}{p}+1}(\mathbb{R}^2))
    \end{align}
    with $\rhoe(t,x) = 1 +\varepsilon \ae(t,x) >0$.
\end{prop} 
The proof of the two propositions above is provided in Appendix \ref{sec:a} and \ref{sec:LWP}.

\begin{proof}[Proof of Theorem \ref{thm}]
Let $2 \leq p<4$ and fix a $\theta=\theta(p)$ so that $p<\theta<4$. 
Let $a_0 \in \fB_{p,1}^{\frac{2}{p}-1}(\mathbb{R}^2) \cap \fB_{p,1}^{\frac{2}{p}}(\mathbb{R}^2)$ and $v_0 \in \fB_{p,1}^{\frac{2}{p}-1}(\mathbb{R}^2)$.
Let $w \in \widetilde{C}([0,\infty);\fB_{p,1}^{\frac{2}{p}-1}(\mathbb{R}^2)) \cap L^1(0,\infty;\fB_{p,1}^{\frac{2}{p}+1}(\mathbb{R}^2))$ be the solution to \eqref{eq:lim-1} with the initial value $\mathbb{P}v_0$.
Let $(\ael,\vel)$ be the solution to 
\begin{align}
    \begin{dcases}
    \partial_t \ael + \frac{1}{\varepsilon} \div \vel = 0, \\
    \partial_t \vel - \Delta  \vel - \kappa \Delta (\varepsilon \nabla \ael) + \frac{1}{\varepsilon} \nabla \ael = -\mathbb{Q} (w \cdot \nabla)w,\\
    \ael(0,x) = a_0(x),\quad \vel(0,x) = \mathbb{Q}v_0(x).
    \end{dcases}
\end{align}
Here, by Lemma \ref{lemm:lin-ene}, we see that for all $1\leq r\leq \infty$ 
\begin{align}
    &
    \n{\sp{\ael,\varepsilon \nabla \ael,\vel}}_{\widetilde{L^{r}}(0,\infty;\fB_{p,1}^{\frac{2}{p}-1+\frac{2}{r}})}\\
    &\quad 
    \leq
    C
    \n{\sp{a_0,\varepsilon \nabla a_0,v_0}}_{\fB_{p,1}^{\frac{2}{p}-1}}
    +
    C
    \n{(w \cdot \nabla)w}_{L^1(0,\infty;\fB_{p,1}^{\frac{2}{p}-1})}\\
    &\quad
    \leq
    C
    \n{\sp{a_0,\varepsilon \nabla a_0,v_0}}_{\fB_{p,1}^{\frac{2}{p}-1}}
    +
    C\n{w}_{ \widetilde{L^{\infty}}(0,\infty;\fB_{p,1}^{\frac{2}{p}-1})
  \cap 
  \widetilde{L^1}(0,\infty;\fB_{p,1}^{\frac{2}{p}+1}) }^2,
\end{align}
which imply
\begin{align}\label{A}
    & 
    \n{\sp{\ael,\varepsilon \nabla \ael, \vel,w}}_{L^r(0,\infty;\dB_{\theta,1}^{\frac{2}{\theta}-1+\frac{2}{r}})}\\
     &\quad
    \leq 
    C 
    \n{\sp{\ael,\varepsilon \nabla \ael, \vel,w}}
    _{L^{r}(0,\infty;\fB_{p,1}^{\frac{2}{p}-1+\frac{2}{r}})} \\ 
     &\quad
    \leq 
    C 
    \n{\sp{\ael,\varepsilon \nabla \ael, \vel,w}}
    _{\widetilde{L^{r}}(0,\infty;\fB_{p,1}^{\frac{2}{p}-1+\frac{2}{r}})} \\ 
    &\quad 
    \leq
    C
    \n{\sp{a_0,\varepsilon \nabla a_0,v_0}}_{\fB_{p,1}^{\frac{2}{p}-1}}
    +  C\n{w}_{ \widetilde{L^{\infty}}(0,\infty;\fB_{p,1}^{\frac{2}{p}-1})
  \cap 
  \widetilde{L^1}(0,\infty;\fB_{p,1}^{\frac{2}{p}+1}) }^2 \\
    &\qquad
    +
     C\n{w}_{ \widetilde{L^{\infty}}(0,\infty;\fB_{p,1}^{\frac{2}{p}-1})
  \cap 
  \widetilde{L^1}(0,\infty;\fB_{p,1}^{\frac{2}{p}+1}) }  \\
    &\quad=:A[a_0,v_0].
\end{align} 
Here, we notice that $A[a_0,v_0]$ depends only on $a_0$ and $v_0$, since $w$ is uniquely determined by $\mathbb{P}v_0$. Let $(\ae,\ve)$ be the local solution to \eqref{eq:NSK-2} in the class 
\begin{align}
    \ae, \varepsilon \nabla \ae, \ve 
    \in 
    \widetilde{C}
    ([0,T_{\varepsilon}^{\rm max});\fB_{p,1}^{\frac{2}{p}-1}(\mathbb{R}^2)) 
    \cap 
    L^1(0,T_{\varepsilon}^{\rm max};\fB_{p,1}^{\frac{2}{p}+1}(\mathbb{R}^2)),
\end{align}
where $T_{\varepsilon}^{\rm max}$ denotes the maximal existence time.
To show the global well-posedness, it suffices to prove the following quantity is infinite:
\begin{align}
    T_{\varepsilon}^*
    :=
    \sup
    \Mp{
    T \in (0,T_{\varepsilon}^{\rm max})\ ;\ 
    \n{\sp{\Ae,\varepsilon \nabla \Ae,\Ve}}_{\widetilde{L^{\infty}}(0,T;\fB_{p,1}^{\frac{2}{p}-1}) \cap L^1(0,T;\fB_{p,1}^{\frac{2}{p}+1})}
    \leq \delta_0
    },
\end{align}
where $\delta_0$ is a positive constant to be determined later, and we have set
$\Ae:=\ae - \ael$ and $\Ve:= \ve - w - \vel$. 
We notice that Proposition \ref{lemm:LWP-CNSK} implies $T_{\varepsilon}^*>0$.
To prove $T_{\varepsilon}^*=\infty$,
we suppose by contradiction that $T_{\varepsilon}^*<\infty$ and let $0<T<T_{\varepsilon}^*$.
The perturbations $(\Ae,\Ve)$ should solve
\begin{align}
    \begin{dcases}
    \partial_t \Ae + \frac{1}{\varepsilon} \div \Ve = - \mathcal{A}\lp{\ae,\ve}, \\
    \partial_t \Ve - \mathcal{L} \Ve - \kappa \Delta (\varepsilon \nabla \Ae) + \frac{1}{\varepsilon} \nabla \Ae = - \sp{\mathcal{V}_{\varepsilon}[\ae,\ve]-(w \cdot \nabla)w},\\
    \Ae(0,x) = 0,\quad 
    \Ve(0,x) = 0.
    \end{dcases}
\end{align}
We decompose nonlinear terms into three parts:
\begin{align}
    &
    \mathcal{A}\lp{\ae,\ve}
    =
    \mathcal{A}_{\varepsilon}^{(0)}
    +
    \mathcal{A}_{\varepsilon}^{(1)}
    +
    \mathcal{A}_{\varepsilon}^{(2)},\\
    &
    \mathcal{V}_{\varepsilon}\lp{\ae,\ve}
    -(w \cdot \nabla)w
    =
    \mathcal{V}_{\varepsilon}^{(0)}
    +
    \mathcal{V}_{\varepsilon}^{(1)}
    +
    \mathcal{V}_{\varepsilon}^{(2)},
\end{align}
where we have set 
\begin{align}
    &
    \mathcal{A}_{\varepsilon}^{(0)}
    :=
    \div (\ael\vel) + \div(\ael w),\\
    &
    \mathcal{A}_{\varepsilon}^{(1)}
    :=
    \div(\ael \Ve) + \div (\Ae \vel) + \div (\Ae w),\\ 
    &
    \mathcal{A}_{\varepsilon}^{(2)}
    :=
    \div (\Ae \Ve),\\
    &
    \begin{aligned}
    \mathcal{V}_{\varepsilon}^{(0)}
    :={}
    &
    (\vel \cdot \nabla)\vel 
    +
    (\vel \cdot \nabla)w 
    +
    (w \cdot \nabla)\vel\\
    &
    +
    \mathcal{J}(\varepsilon \ael)
    \mathcal{L}w
    +
    \mathcal{J}(\varepsilon \ael)
    \mathcal{L}\vel
    +
    \frac{1}{\varepsilon}
    \mathcal{K}(\varepsilon \ael)\nabla \ael ,
    \end{aligned}
    \\
    &
    \begin{aligned}
    \mathcal{V}_{\varepsilon}^{(1)}
    :=
    &
    (\Ve \cdot \nabla)\vel 
    +
    (\Ve \cdot \nabla)w 
    +
    (\vel \cdot \nabla)\Ve
    +
    (w \cdot \nabla)\Ve\\
    &
    +
    \sp{\mathcal{J}(\varepsilon \ae)-\mathcal{J}(\varepsilon \ael)}
    \mathcal{L}w
    +
    \sp{\mathcal{J}(\varepsilon \ae)-\mathcal{J}(\varepsilon \ael)}
    \mathcal{L}\vel\\
    &+
    \mathcal{J}(\varepsilon \ael)
    \mathcal{L}\Ve
    +
    \frac{1}{\varepsilon}
    \sp{\mathcal{K}(\varepsilon \ae) - \mathcal{K}(\varepsilon \ael)}
    \nabla \ael
    +
    \frac{1}{\varepsilon}
    \mathcal{K}(\varepsilon \ael)\nabla \Ae,
    \end{aligned}
    \\
    &
    \mathcal{V}_{\varepsilon}^{(2)}
    :=
    (\Ve \cdot \nabla)\Ve
    +
    \sp{\mathcal{J}(\varepsilon \ae)-\mathcal{J}(\varepsilon \ael)}
    \mathcal{L}     
    \Ve 
    +
    \frac{1}{\varepsilon}
    \sp{\mathcal{K}(\varepsilon \ae) - \mathcal{K}(\varepsilon \ael)}
    \nabla \Ae.
\end{align}
It follows from Lemma \ref{lemm:lin-ene} that 
\begin{align}\label{p}
    &\n{\sp{\Ae,\varepsilon \nabla \Ae,\Ve}}_{\widetilde{L^{\infty}}(0,T;\fB_{p,1}^{\frac{2}{p}-1}) \cap L^1(0,T;\fB_{p,1}^{\frac{2}{p}+1})}
    \\
    &\quad
    \leq 
    C
    \sum_{m=0,1,2}
    \int_0^T
    \n{\sp{\mathcal{A}_{\varepsilon}^{(m)},\varepsilon \nabla \mathcal{A}_{\varepsilon}^{(m)}, \mathcal{V}_{\varepsilon}^{(m)}}(t)}_{\fB_{p,1}^{\frac{2}{p}-1}}
    dt
    =:
    C 
    \sum_{m=0,1,2}
    \mathcal{I}_{\varepsilon}^{(m)}.
\end{align}

We first consider the estimate of $\mathcal{I}_{\varepsilon}^{(0)}$.
We see from the continuous embedding $\dot{B}^{0}_{2,1}(\R^2) \hookrightarrow \fB_{p,1}^{\f{2}{p}-1}(\R^2)$ that 
\begin{align}
    &
    \n{\varepsilon \nabla \div (\ael\vel)}_{L^1(0,T;\fB_{p,1}^{\frac{2}{p}-1})}\\
    &\quad\leq{}
    C
    \n{\varepsilon \nabla \div T_{\ael}\vel}_{L^1(0,T;\dB_{2,1}^{0})}^{\ell;\frac{1}{\varepsilon}}
    +
    C
    \n{\varepsilon \nabla \div T_{\ael}\vel}_{L^1(0,T;\fB_{p,1}^{\frac{2}{p}-1})}^{h;\frac{1}{\varepsilon}}\\
    &\qquad
    +
    C
    \n{\varepsilon \nabla \div R(\ael,\vel)}_{L^1(0,T;\dB_{2,1}^{0})}
    +
    C
    \n{\varepsilon \nabla \div T_{\vel}\ael}_{L^1(0,T;\dB_{2,1}^{0})}.
\end{align} 
By Lemma \ref{lemm-Tfg-Rfg}, Lemma \ref{lemm-Tfg-low}, the continuous embedding $\dot{B}_{\theta,1}^{\f{2}{\theta}-\f{1}{2}  }(\R^2) \hookrightarrow \dot{B}_{\f{2\theta}{\theta-2},1}^{ \f{1}{2}- \f{2}{\theta}  } (\R^2)  $ and the fact that $\n{f}_{ \dot{B}_{p,1}^{s+1} }^{\ell;\f{4}{\varepsilon}} \leq (C/\varepsilon) \n{f}_{ \dot{B}_{p,1}^{s} }  $ for all $1\leq p \leq \infty$ and $s  \in \mathbb{R}$, it holds
\begin{align}
    &
    \begin{aligned}
    \n{\varepsilon \nabla \div T_{\ael}\vel}_{L^1(0,T;\dB_{2,1}^{0})}^{\ell;\frac{1}{\varepsilon}}
    \leq{}&
    C\varepsilon
    \n{\ael}_{L^4(0,T;\dB_{\theta,1}^{\frac{2}{\theta}-\frac{1}{2}})}
    \n{\vel}_{L^{\frac{4}{3}}(0,T;\dB_{\theta,1}^{\frac{2}{\theta}+\frac{3}{2}})}^{\ell;\frac{4}{\varepsilon}}\\
    \leq{}&
    C
    \n{\ael}_{L^4(0,T;\dB_{\theta,1}^{\frac{2}{\theta}-\frac{1}{2}})}
    \n{\vel}_{L^{\frac{4}{3}}(0,T;\dB_{\theta,1}^{\frac{2}{\theta}+\frac{1}{2}})},
    \end{aligned}
    \\
    &
    \begin{aligned}
    &
    \n{\varepsilon \nabla \div R(\ael,\vel)}_{L^1(0,T;\dB_{2,1}^{0})}
    +
    \n{\varepsilon \nabla \div T_{\vel}\ael}_{L^1(0,T;\dB_{2,1}^{0})}\\
    &\quad
    \leq 
    C
    \n{\vel}_{L^4(0,T;\dB_{\theta,1}^{\frac{2}{\theta}-\frac{1}{2}})}
    \n{\varepsilon \nabla \ael}_{L^{\frac{4}{3}}(0,T;\dB_{\theta,1}^{\frac{2}{\theta}+\frac{1}{2}})} .
    \end{aligned}
\end{align} 
Due to Lemma \ref{lemm-Tfg-low} and the continuous embedding $\fB_{p,1}^{\f{2}{p}}(\R^2) \hookrightarrow \fB_{\infty,1}^{0}(\R^2)$, 
 \begin{align}
    \n{\varepsilon \nabla \div T_{\ael}\vel}_{L^1(0,T;\fB_{p,1}^{\frac{2}{p}-1})}^{h;\frac{1}{\varepsilon}}
    \leq{}&
    C
    \n{\varepsilon \nabla \ael}_{L^{\infty}(0,T;\fB_{p,1}^{\frac{2}{p}-1})}
    \n{\vel}_{L^1(0,T;\fB_{p,1}^{\frac{2}{p}+1})}^{h;\frac{1}{4\varepsilon}}.
    \end{align}
Thus, we have
\begin{align}
    &
    \n{\varepsilon \nabla \div (\ael\vel)}_{L^1(0,T;\fB_{p,1}^{\frac{2}{p}-1})}
    +
    \n{\varepsilon \nabla \div (\ael w)}_{L^1(0,T;\fB_{p,1}^{\frac{2}{p}-1})}\\
    &\quad\leq{}
    C
    \n{\varepsilon \nabla \ael}_{L^{\infty}(0,T;\fB_{p,1}^{\frac{2}{p}-1})}
    \n{\sp{\vel,w}}_{L^1(0,T;\fB_{p,1}^{\frac{2}{p}+1})}^{h;\frac{1}{4\varepsilon}}\\
    &\qquad
    +
    C
    \n{\sp{\vel,w}}_{L^4(0,T;\dB_{\theta,1}^{\frac{2}{\theta}-\frac{1}{2}})\cap L^{\frac{4}{3}}(0,T;\dB_{\theta,1}^{\frac{2}{\theta}+\frac{1}{2}})}\\
    &\qquad \qquad \times
    \n{\sp{\ael,\varepsilon \nabla \ael}}_{L^4(0,T;\dB_{\theta,1}^{\frac{2}{\theta}-\frac{1}{2}})\cap L^{\frac{4}{3}}(0,T;\dB_{\theta,1}^{\frac{2}{\theta}+\frac{1}{2}})}.
\end{align}
Similarly, we have 
\begin{align}
    &
    \n{\mathcal{J}(\varepsilon \ael)\mathcal{L}\vel}_{L^1(0,T;\fB_{p,1}^{\frac{2}{p}-1})}
    +
    \n{\mathcal{J}(\varepsilon \ael)\mathcal{L}w}_{L^1(0,T;\fB_{p,1}^{\frac{2}{p}-1})}
    \\
    &\quad\leq{}
    C
    \n{\varepsilon \nabla \ael}_{L^{\infty}(0,T;\fB_{p,1}^{\frac{2}{p}-1})}
    \n{\sp{\vel,w}}_{L^1(0,T;\fB_{p,1}^{\frac{2}{p}+1})}^{h;\frac{1}{4\varepsilon}}\\
    &\qquad
    +
    C
    \n{\sp{\vel,w}}_{L^4(0,T;\dB_{\theta,1}^{\frac{2}{\theta}-\frac{1}{2}})\cap L^{\frac{4}{3}}(0,T;\dB_{\theta,1}^{\frac{2}{\theta}+\frac{1}{2}})}\\
    &\qquad \qquad \times
    \n{\sp{\ael,\varepsilon \nabla \ael}}_{L^4(0,T;\dB_{\theta,1}^{\frac{2}{\theta}-\frac{1}{2}})\cap L^{\frac{4}{3}}(0,T;\dB_{\theta,1}^{\frac{2}{\theta}+\frac{1}{2}})}.
\end{align} 
To handle the other terms, we invoke the following product estimate in Besov spaces from \cite{Fuj-23-05}*{Lemma 4.3}, where it is stated that for $2\leq q< 4$ and $q\leq r \leq \infty $,
there exists a positive constant $C=C(q,r)$ such that 
\begin{align}
  \n{ f \cdot \Grad g}_{L^1(0,T;\dB_{2,1}^0)}
& \leq
 C
 \n{g}_{L^r (0,T;\dB_{q,1}^{ \f{2}{q}-1+\f{2}{r}  })}
  \n{f}_{L^{r'}(0,T;\dB_{q,1}^{ \f{2}{q}-1+\f{2}{r'}  })} \\
  & \quad 
  +
  C
  \n{g}_{L^{r'}(0,T;\dB_{q,1}^{ \f{2}{q}-1+\f{2}{r'}  })}
   \n{g}_{L^r (0,T;\dB_{q,1}^{ \f{2}{q}-1+\f{2}{r}  })}. \label{est-f-grad-g} 
\end{align}
Employing the above estimate with $q=\theta,r=4$ readily yields  
\begin{align}
    &
    \begin{aligned}
    &
    \n{\div (\ael\vel)
    + 
    \div(\ael w)}_{L^1(0,T;\fB_{p,1}^{\frac{2}{p}-1})}\\
    &\quad
    \leq
    C
    \n{\div (\ael\vel)}_{L^1(0,T;\dB_{2,1}^0)}
    +
    C
    \n{\div(\ael w)}_{L^1(0,T;\dB_{2,1}^0)}\\
    &\quad
    \leq
    C
    \n{\ael}_{L^4(0,T;\dB_{\theta,1}^{\frac{2}{\theta}-\frac{1}{2}})\cap L^{\frac{4}{3}}(0,T;\dB_{\theta,1}^{\frac{2}{\theta}+\frac{1}{2}})}
    \n{\sp{\vel,w}}_{L^4(0,T;\dB_{\theta,1}^{\frac{2}{\theta}-\frac{1}{2}})\cap L^{\frac{4}{3}}(0,T;\dB_{\theta,1}^{\frac{2}{\theta}+\frac{1}{2}})},
    \end{aligned}\\
    &
    \begin{aligned}
    &
    \n{(\vel \cdot \nabla)\vel 
    +
    (\vel \cdot \nabla)w 
    +
    (w \cdot \nabla)\vel}_{L^1(0,T;\fB_{p,1}^{\frac{2}{p}-1})}\\
    &\quad\leq
    C
    \n{(\vel \cdot \nabla)\vel 
    +
    (\vel \cdot \nabla)w 
    +
    (w \cdot \nabla)\vel}_{L^1(0,T;\dB_{2,1}^0)}\\
    &\quad
    \leq
    C
    \n{\vel}_{L^4(0,T;\dB_{\theta,1}^{\frac{2}{\theta}-\frac{1}{2}})\cap L^{\frac{4}{3}}(0,T;\dB_{\theta,1}^{\frac{2}{\theta}+\frac{1}{2}})}
    \n{\sp{\vel,w}}_{L^4(0,T;\dB_{\theta,1}^{\frac{2}{\theta}-\frac{1}{2}})\cap L^{\frac{4}{3}}(0,T;\dB_{\theta,1}^{\frac{2}{\theta}+\frac{1}{2}})},
    \end{aligned}\\
    &
    \begin{aligned}
    \n{\frac{1}{\varepsilon}
    \mathcal{K}(\varepsilon \ael)\nabla \ael}_{L^1(0,T;\fB_{p,1}^{\frac{2}{p}-1})}
    &\leq
    C
    \n{\frac{1}{\varepsilon}
    \mathcal{K}(\varepsilon \ael)\nabla \ael}_{L^1(0,T;\dB_{2,1}^0)}\\
    &\leq
    C
    \n{\ael}_{L^4(0,T;\dB_{\theta,1}^{\frac{2}{\theta}-\frac{1}{2}})\cap L^{\frac{4}{3}}(0,T;\dB_{\theta,1}^{\frac{2}{\theta}+\frac{1}{2}})}^2,
    \end{aligned}
\end{align}   
where we also used the composition lemma in the context of Besov space in the last step. Combining the above estimates and using \eqref{A}
for all $1 \leq r \leq \infty$,
we have 
\begin{align}
    \mathcal{I}_{\varepsilon}^{(0)}
    \leq{}&
    CA[a_0,v_0]
    \n{\sp{\vel,w}}_{L^1(0,T;\fB_{p,1}^{\frac{2}{p}+1})}^{h;\frac{1}{4\varepsilon}}\\
    &
    +
    CA[a_0,v_0]
    {\n{\sp{\ael,\vel}}
    _{
    L^4(0,T;\dB_{\theta,1}^{\frac{2}{\theta}-\frac{1}{2}})
    \cap 
    L^{\frac{4}{3}}(0,T;\dB_{\theta,1}^{\frac{2}{\theta}+\frac{1}{2}})}}.
\end{align}

Next, we focus on the estimates of $\mathcal{I}_{\varepsilon}^{(1)}$. Using the continuous embedding $\fB_{p,1}^{\f{2}{p}}(\R^2) \hookrightarrow \widehat{L^{\infty}}(\R^2)$ and the following product estimate for all $1\leq p,\sigma\leq \infty$ and $s>0$
\begin{align}
    \n{fg}_{\fB^s_{p,\sigma}} 
    \leq 
    C 
    \left( 
    \n{f}_{\widehat{L^{\infty}}}
    \n{g}_{\fB^s_{p,\sigma}}
    +
    \n{f}_{\fB^s_{p,\sigma}}
    \n{g}_{\widehat{L^{\infty}}}
    \right), 
\end{align}
one sees that 
\begin{align}
  &  \n{\varepsilon \nabla \div \sp{ \ael  \Ve}}_{L^1(0,T;\fB_{p,1}^{\frac{2}{p}-1})} 
    \leq
    C
    \n{\varepsilon \nabla \ael}_{L^{\infty}(0,T;\fB_{p,1}^{\frac{2}{p}-1})}
    \n{\Ve}_{L^1(0,T;\fB_{p,1}^{\frac{2}{p}+1})}\\
    & \quad \quad \quad \quad \quad \quad \quad \quad
    \quad \quad \quad \quad \quad \quad
    +
    C
    \int_0^T
    \n{\varepsilon \nabla \ael(t)}_{\fB_{p,1}^{\frac{2}{p}}}
    {\n{\Ve(t)}_{\fB_{p,1}^{\frac{2}{p}}}}
    dt, \\
  & \n{\varepsilon \nabla \div \sp{\Ae \sp{\vel + w}}}_{L^1(0,T;\fB_{p,1}^{\frac{2}{p}-1})}  \\
   & \quad \leq 
    C
    \int_0^T
    \n{\varepsilon \nabla \Ae(t)}_{\fB_{p,1}^{\frac{2}{p}-1}}
    {\n{\sp{\vel,w}(t)}_{\fB_{p,1}^{\frac{2}{p}+1}}}
    dt \\
    &\quad \quad 
    +
    C
    \int_0^T
    \n{\varepsilon \nabla \Ae(t)}_{\fB_{p,1}^{\frac{2}{p}}}
    {\n{\sp{\vel,w}(t)}_{\fB_{p,1}^{\frac{2}{p}}}}
    dt,             \\
    & \n{ \div(\ael \Ve) 
    +
    \div(\Ae \vel)
    +
    \div(\Ae w)
    }_{L^1(0,T;\fB_{p,1}^{\frac{2}{p}-1})} \\
    & \quad 
    \leq 
    C
    \int_0^T
      \n{ (\Ae,\Ve)(t)   }_{\fB_{p,1}^{\frac{2}{p}}} 
      \n{ (\ael,\vel,w  )(t)  }_{\fB_{p,1}^{\frac{2}{p}}}
    dt.
\end{align}
It follows from Lemma \ref{lemm-fg-general} with $p_1=p_2=p,s=2/p-1,(\alpha_1,\alpha_2)=(0,1)$ that\footnote{We emphasize that, in applying Lemma \ref{lemm-fg-general}, the assumption $s+2/p_1>0$ requires that $p$ is strictly less than $4$. } 
 \begin{align}
    &
    \n{(\Ve \cdot \nabla)\vel + (\Ve \cdot \nabla)w + (w \cdot \nabla)\Ve + (\vel \cdot \nabla)\Ve}_{L^1(0,T;\fB_{p,1}^{\frac{2}{p}-1})}\\
    &\quad
    \leq
    C
    \int_0^T
    \n{\sp{\vel,w}(t)}_{\fB_{p,1}^{\frac{2}{p}}}
    \n{\Ve(t)}_{\fB_{p,1}^{\frac{2}{p}}}
    dt. 
\end{align} 
In the same spirit as above, we obtain with the aid of the composition Lemma \ref{lemm-comp} that   
\begin{align}
    &
    \begin{aligned}
    &
    \n{\sp{\mathcal{J}(\varepsilon \ae) - \mathcal{J}(\varepsilon \ael)}\mathcal{L}\sp{\vel+w}}_{L^1(0,T;\fB_{p,1}^{\frac{2}{p}-1})}\\
    &\quad
    \leq
    C
    \int_0^T
    \n{\varepsilon \nabla \Ae(t)}_{\fB_{p,1}^{\frac{2}{p}-1}}
    \n{\sp{\vel,w}(t)}_{\fB_{p,1}^{\frac{2}{p}+1}}
    dt\\
    &\qquad
    +
    C
    \int_0^T
    \n{\Ae(t)}_{\fB_{p,1}^{\frac{2}{p}}}
    \n{\sp{\vel,w}(t)}_{\fB_{p,1}^{\frac{2}{p}}}
    dt,
    \end{aligned}\\
    &
    \begin{aligned}
    \n{\mathcal{J}(\varepsilon \ael)\mathcal{L}\Ve}_{L^1(0,T;\fB_{p,1}^{\frac{2}{p}-1})}
    \leq{}&
    C
    \n{\varepsilon \nabla \ael}_{L^{\infty}(0,T;\fB_{p,1}^{\frac{2}{p}-1})}
    \n{\Ve}_{L^1(0,T;\fB_{p,1}^{\frac{2}{p}+1})}\\
    &
    +
    C
    \int_0^T
    \n{\ael(t)}_{\fB_{p,1}^{\frac{2}{p}}}
    \n{\Ve(t)}_{\fB_{p,1}^{\frac{2}{p}}}
    dt,
    \end{aligned}\\
    &
    \begin{aligned}
    &
    \n{
    \frac{1}{\varepsilon}
    \sp{\mathcal{K}(\varepsilon \ae) - \mathcal{K}(\varepsilon \ael)}
    \nabla \ael
    +
    \frac{1}{\varepsilon}
    \mathcal{K}(\varepsilon \ael)\nabla \Ae
    }_{L^1(0,T;\fB_{p,1}^{\frac{2}{p}-1})}\\
    &\quad
    \leq
    C
    \int_0^T
    \n{\ael(t)}_{\fB_{p,1}^{\frac{2}{p}}}
    \n{\Ae(t)}_{\fB_{p,1}^{\frac{2}{p}}}
    dt
    \end{aligned}
\end{align} 
with the help of Young's inequality and the following interpolation inequality
\begin{align}
    \n{ f }_{\fB_{p,1}^{\frac{2}{p} }} 
    \leq{}
    \n{f }_{\fB_{p,1}^{\frac{2}{p}-1}}^{\frac{1}{2}}
    \n{ f }_{\fB_{p,1}^{\frac{2}{p}+1}}^{\frac{1}{2}}.
\end{align} 
Thus, we have
\begin{align}\label{I1}
    \mathcal{I}_{\varepsilon}^{(1)}
    \leq {}&
    C_1
    \n{\varepsilon \nabla \ael}_{L^{\infty}(0,T;\fB_{p,1}^{\frac{2}{p}-1})}
    \n{\Ve}_{L^1(0,T;\fB_{p,1}^{\frac{2}{p}+1})}\\
    &
    +
    \frac{1}{4} 
    \int_0^T
    \n{\sp{\Ae,\varepsilon\nabla\Ae,\Ve}(t)}_{\fB_{p,1}^{\frac{2}{p}+1}}
    dt\\
    &
    +
    C_1  
    \int_0^T
    \n{\sp{\Ae,\varepsilon\nabla\Ae,\Ve}(t)}_{\fB_{p,1}^{\frac{2}{p}-1}}
    \sum_{r=1,\frac{4}{3},2,4}\n{\sp{\ael,\varepsilon \nabla  \ael ,\vel,w}(t)}_{\fB_{p,1}^{\frac{2}{p}-1+\frac{2}{r}}}^r
    dt
\end{align}
for some positive constant $C_1$.
By Lemma \ref{lemm:lin-0-1} and the continuous embedding $\widetilde{L^{\infty}} (0,T;\fB_{p,1}^{\frac{2}{p}-1} (\R^2)) \hookrightarrow  L^{\infty}(0,T;\fB_{p,1}^{\frac{2}{p}-1}(\R^2))$, there exists a positive constant $\varepsilon_1$ such that 
\begin{align}
    \n{\varepsilon \nabla \ael}_{ L^{\infty}(0,T;\fB_{p,1}^{\frac{2}{p}-1}) }\leq \frac{1}{4C_1}
\end{align}
for all $0<\varepsilon\leq\varepsilon_1$.
In the following of this proof, we assume $0<\varepsilon\leq\varepsilon_1$.
Then, \eqref{I1} is rewritten as 
\begin{align}
    \mathcal{I}_{\varepsilon}^{(1)}
    \leq{}& 
    \frac{1}{2} 
    \int_0^T
    \n{\sp{\Ae,\varepsilon\nabla\Ae,\Ve}(t)}_{\fB_{p,1}^{\frac{2}{p}+1}}
    dt\\
    &
    +
    C_1
    \int_0^T
    \n{\sp{\Ae,\varepsilon\nabla\Ae,\Ve}(t)}_{\fB_{p,1}^{\frac{2}{p}-1}}
    \sum_{r=1,\frac{4}{3},2,4}\n{\sp{\ael,\varepsilon \nabla \ael ,\vel,w}(t)}_{\fB_{p,1}^{\frac{2}{p}-1+\frac{2}{r}}}^r
    dt.
\end{align}

For the estimate of $\mathcal{I}_{\varepsilon}^{(2)}$, we have
\begin{align}
    &
    \begin{aligned}
        \n{\div(\Ae\Ve)}_{L^1(0,T;\fB_{p,1}^{\frac{2}{p}-1})}
        \leq
        C
        \n{\Ae}_{L^2(0,T;\fB_{p,1}^{\frac{2}{p}})}
        \n{\Ve}_{L^2(0,T;\fB_{p,1}^{\frac{2}{p}})},
    \end{aligned}
    \\
    &
    \begin{aligned}
        \n{\varepsilon \nabla \div(\Ae\Ve)}_{L^1(0,T;\fB_{p,1}^{\frac{2}{p}-1})}
        \leq{}&
        C
        \n{\varepsilon \nabla \Ae}_{L^{\infty}(0,T;\fB_{p,1}^{\frac{2}{p}-1})}
        \n{\Ve}_{L^1(0,T;\fB_{p,1}^{\frac{2}{p}+1})}\\
        &
        +
        C
        \n{\varepsilon \nabla \Ae}_{L^2(0,T;\fB_{p,1}^{\frac{2}{p}})}
        \n{\Ve}_{L^2(0,T;\fB_{p,1}^{\frac{2}{p}})},
    \end{aligned}
    \\
    &
    \begin{aligned}
        \n{(\Ve \cdot \nabla)\Ve}_{L^1(0,T;\fB_{p,1}^{\frac{2}{p}-1})}
        \leq{}&
        C
        \n{\Ve}_{L^2(0,T;\fB_{p,1}^{\frac{2}{p}})}^2,
    \end{aligned}
    \\
    &
    \begin{aligned}
        \n{(\mathcal{J}(\varepsilon\ae)-\mathcal{J}(\varepsilon\ael))\mathcal{L}\Ve}_{L^1(0,T;\fB_{p,1}^{\frac{2}{p}-1})}
        \leq{}&
        C
        \n{\varepsilon \nabla \Ae}_{L^{\infty}(0,T;\fB_{p,1}^{\frac{2}{p}-1})}
        \n{\Ve}_{L^1(0,T;\fB_{p,1}^{\frac{2}{p}+1})}\\
        &
        +
        C
        \n{\varepsilon \nabla \Ae}_{L^2(0,T;\fB_{p,1}^{\frac{2}{p}})}
        \n{\Ve}_{L^2(0,T;\fB_{p,1}^{\frac{2}{p}})},
    \end{aligned}
    \\
    &
    \begin{aligned}
        \n{\frac{1}{\varepsilon}(\mathcal{K}(\varepsilon \ae)-\mathcal{K}(\varepsilon \ael))\nabla\Ae}_{L^1(0,T;\fB_{p,1}^{\frac{2}{p}-1})}
        \leq{}&
        C
        \n{\Ae}_{L^2(0,T;\fB_{p,1}^{\frac{2}{p}})}^2.
    \end{aligned}
\end{align}
Combining the above estimates, we see that 
\begin{align}
    \mathcal{I}_{\varepsilon}^{(2)}
    \leq
    C
    \sp{
    \n{\sp{\Ae,\varepsilon \nabla \Ae,\Ve}}_{\widetilde{L^{\infty}}(0,T;\fB_{p,1}^{\frac{2}{p}-1}) \cap L^1(0,T;\fB_{p,1}^{\frac{2}{p}+1})}
    }^2.
\end{align}

Hence, we obtain 
\begin{align}
    &
    \n{\sp{\Ae,\varepsilon \nabla \Ae,\Ve}}_{\widetilde{L^{\infty}}(0,T;\fB_{p,1}^{\frac{2}{p}-1}) \cap L^1(0,T;\fB_{p,1}^{\frac{2}{p}+1})}
    \leq{}
    C_2A[a_0,v_0]
    \n{\sp{\vel,w}}_{L^1(0,T;\fB_{p,1}^{\frac{2}{p}+1})}^{h;\frac{1}{\varepsilon}}\\
    &\qquad 
    +
    C_2A[a_0,v_0]
    {\n{\sp{\ael,\vel}}
    _{
    L^4(0,\infty;\dB_{\theta,1}^{\frac{2}{\theta}-\frac{1}{2}})
    \cap 
    L^{\frac{4}{3}}(0,\infty;\dB_{\theta,1}^{\frac{2}{\theta}+\frac{1}{2}})}}\\
    &\qquad
    +
    C_2
    \int_0^T
    \n{\sp{\Ae,\varepsilon\nabla\Ae,\Ve}(t)}_{\fB_{p,1}^{\frac{2}{p}-1}}
    \sum_{r=1,\frac{4}{3},2,4}\n{\sp{\ael,\varepsilon \nabla  \ael ,\vel,w}(t)}_{\fB_{p,1}^{\frac{2}{p}-1+\frac{2}{r}}}^r
    dt\\
    &\qquad
    +
    C_2\sp{
    \n{\sp{\Ae,\varepsilon \nabla \Ae,\Ve}}_{\widetilde{L^{\infty}}(0,T;\fB_{p,1}^{\frac{2}{p}-1}) \cap L^1(0,T;\fB_{p,1}^{\frac{2}{p}+1})}
    }^2.
\end{align}
Let $\delta_0:=1/(2C_2)$.
Then, we see that 
\begin{align}
    &
    \n{\sp{\Ae,\varepsilon \nabla \Ae,\Ve}}_{\widetilde{L^{\infty}}(0,T;\fB_{p,1}^{\frac{2}{p}-1}) \cap L^1(0,T;\fB_{p,1}^{\frac{2}{p}+1})}
    \leq{}
    2C_2A[a_0,v_0]
    \n{\sp{\vel,w}}_{L^1(0,T;\fB_{p,1}^{\frac{2}{p}+1})}^{h;\frac{1}{\varepsilon}}\\
    &\qquad
    +
    2C_2A[a_0,v_0]
    {\n{\sp{\ael,\vel}}
    _{
    L^4(0,\infty;\dB_{\theta,1}^{\frac{2}{\theta}-\frac{1}{2}})
    \cap 
    L^{\frac{4}{3}}(0,\infty;\dB_{\theta,1}^{\frac{2}{\theta}+\frac{1}{2}})}}\\
    &\qquad
    +
    2C_2
    \int_0^T
    \n{\sp{\Ae,\varepsilon\nabla\Ae,\Ve}(t)}_{\fB_{p,1}^{\frac{2}{p}-1}}
    \sum_{r=1,\frac{4}{3},2,4}\n{\sp{\ael,\varepsilon \nabla  \ael ,\vel,w}(t)}_{\fB_{p,1}^{\frac{2}{p}-1+\frac{2}{r}}}^r
    dt.
\end{align}
It follows from the Gronwall lemma that
\begin{align}
    &
    \n{\sp{\Ae,\varepsilon \nabla \Ae,\Ve}}_{\widetilde{L^{\infty}}(0,T;\fB_{p,1}^{\frac{2}{p}-1}) \cap L^1(0,T;\dB_{p,1}^{\frac{2}{p}+1})}\\
    &\leq{}
    CA[a_0,v_0]
    \sp{
    \n{\sp{\vel,w}}_{L^1(0,T;\fB_{p,1}^{\frac{2}{p}+1})}^{h;\frac{1}{\varepsilon}}
    + 
    {\n{\sp{\ael,\vel}}
    _{
    L^4(0,\infty;\dB_{\theta,1}^{\frac{2}{\theta}-\frac{1}{2}})
    \cap 
    L^{\frac{4}{3}}(0,\infty;\dB_{\theta,1}^{\frac{2}{\theta}+\frac{1}{2}})}}
    }\\
    &\quad
    \times 
    \exp 
    \lp{
    C  
    \sum_{r=1,\frac{4}{3},2,4}\n{\sp{\ael,\varepsilon \nabla  \ael ,\vel,w}}_{L^r(0,\infty;\fB_{p,1}^{\frac{2}{p}-1+\frac{2}{r}})}^r
    }.
\end{align}
Thus, it follows from \eqref{A} that 
\begin{align}\label{AV}
    &
    \n{\sp{\Ae,\varepsilon \nabla \Ae,\Ve}}_{\widetilde{L^{\infty}}(0,T;\fB_{p,1}^{\frac{2}{p}-1}) \cap L^1(0,T;\fB_{p,1}^{\frac{2}{p}+1})}\\
    &\quad
    \leq{}
    C_3A[a_0,v_0]
    e^{C_3A[a_0,v_0]}
    \n{\sp{\vel,w}}_{L^1(0,\infty;\fB_{p,1}^{\frac{2}{p}+1})}^{h;\frac{1}{\varepsilon}}\\
    &\qquad 
    +
    C_3A[a_0,v_0]
    e^{C_3A[a_0,v_0]}
    {\n{\sp{\ael,\vel}}
    _{
    L^4(0,\infty;\dB_{\theta,1}^{\frac{2}{\theta}-\frac{1}{2}})
    \cap 
    L^{\frac{4}{3}}(0,\infty;\dB_{\theta,1}^{\frac{2}{\theta}+\frac{1}{2}})}}.
\end{align}
for some positive constant $C_3$.
Using Lemma \ref{lemm:lin-0-2}, the embeddings $ \widetilde{L^4}(0,\infty;\dB_{\theta,1}^{\frac{2}{\theta}-\frac{1}{2}})  \hookrightarrow L^4(0,\infty;\dB_{\theta,1}^{\frac{2}{\theta}-\frac{1}{2}})$, 
$\widetilde{  L^{\frac{4}{3}} }(0,\infty;\dB_{\theta,1}^{\frac{2}{\theta}+\frac{1}{2}}) \hookrightarrow L^{\frac{4}{3}}(0,\infty;\dB_{\theta,1}^{\frac{2}{\theta}+\frac{1}{2}}) $ and 
\begin{align}\label{vel-high}
    &
    \n{\sp{\vel,w}}_{L^1(0, \infty ;\fB_{p,1}^{\frac{2}{p}+1})}^{h;\frac{1}{\varepsilon}}\\
    &\quad\leq{}
    \n{{w}}_{L^1(0, \infty ;\fB_{p,1}^{\frac{2}{p}+1})}^{h;\frac{1}{\varepsilon}}
    +
    \n{\sp{a_0,\varepsilon\nabla a_0,v_0}}_{\fB_{p,1}^{\frac{2}{p}-1}}^{h;\frac{1}{\varepsilon}}
    +
    \n{(w\cdot \nabla)w}_{L^1(0,\infty;\fB_{p,1}^{\frac{2}{p}-1})}^{h;\frac{1}{\varepsilon}}\\
    &\quad\to{} 0 \qquad {\rm as }\quad \varepsilon \downarrow 0,
\end{align}
we may choose $0<\varepsilon_2\leq\varepsilon_1$ such that 
\begin{align}
    &
    C_3A[a_0,v_0]
    e^{C_3A[a_0,v_0]}
    \n{\sp{\vel,w}}_{L^1(0,\infty;\fB_{p,1}^{\frac{2}{p}+1})}^{h;\frac{1}{\varepsilon}}\\
    &\quad 
    +
    C_3A[a_0,v_0]
    e^{C_3A[a_0,v_0]}
    {\n{\sp{\ael,\vel}}
    _{
    L^4(0,\infty;\dB_{\theta,1}^{\frac{2}{\theta}-\frac{1}{2}})
    \cap 
    L^{\frac{4}{3}}(0,\infty;\dB_{\theta,1}^{\frac{2}{\theta}+\frac{1}{2}})    }  }
    \leq
    \frac{\delta_0}{2}
\end{align}
for all $0<\varepsilon \leq \varepsilon_2$.
In the following, we assume $0<\varepsilon \leq \varepsilon_2$.
Then, since  
\begin{align}
    \n{\sp{\Ae,\varepsilon \nabla \Ae,\Ve}}_{\widetilde{L^{\infty}}(0,T;\fB_{p,1}^{\frac{2}{p}-1}) \cap L^1(0,T;\fB_{p,1}^{\frac{2}{p}+1})} \leq \frac{\delta_0}{2}
\end{align}
for all $0<T<T_{\varepsilon}^*$, we meet a contradiction to the definition of $T_{\varepsilon}^*$ due to the assumption that $T_{\varepsilon}^*<\infty$.
Hence, we obtain $T_{\varepsilon}^{\rm max} \geq T_{\varepsilon}^*=\infty$ and complete the proof of the global well-posedness.

Finally, we consider the low Mach number limit.
Let $p<q\leq \infty$ and $1<r<\infty$.
Then, we see
by $\ae = \Ae+\ael$, $\mathbb{Q}\ve=\mathbb{Q}\Ve+\vel$, and $\mathbb{P}\ve-w=\mathbb{P}\Ve$
that  
\begin{align}
    & 
    \n{\sp{\ae,\varepsilon \nabla \ae,\mathbb{Q}\ve}}_{\widetilde{L^r}(0,\infty;\dB_{ q,1}^{\frac{2}{q}-1+\frac{2}{r}})}
    +
    \n{\mathbb{P}\ve-w}_{\widetilde{L^{\infty}}(0,\infty;\fB_{p,1}^{\frac{2}{p}-1}) \cap L^1(0,\infty;\fB_{p,1}^{\frac{2}{p}+ 1})}  
    \\
    &\quad
    \leq
    \n{\sp{\ael,\varepsilon \nabla \ael, 
           \vel}}_{\widetilde{L^r}(0,\infty;\dB_{q,1}^{\frac{2}{q}-1+\frac{2}{r}})}\\
    &\qquad+
    \n{\sp{\Ae,\varepsilon \nabla \Ae,\mathbb{Q}\Ve}}_{\widetilde{L^r}(0,\infty;\dB_{q,1}^{\frac{2}{q}-1+\frac{2}{r}})}
    +
    \n{\mathbb{P}\Ve}_{\widetilde{L^{\infty}}(0,\infty;\fB_{p,1}^{\frac{2}{p}-1}) \cap L^1(0,\infty;\fB_{p,1}^{\frac{2}{p}+1})}
    \\
    &\quad
    \leq
    \n{\sp{\ael,\varepsilon \nabla \ael, 
      \vel}}_{\widetilde{L^r}(0,\infty;\dB_{q,1}^{\frac{2}{q}-1+\frac{2}{r}})}
    +
    C
    \n{\sp{\Ae,\varepsilon \nabla  \Ae,\Ve}}_{\widetilde{L^{\infty}}(0,\infty;\fB_{p,1}^{\frac{2}{p}-1}) \cap L^1(0,\infty;\fB_{p,1}^{\frac{2}{p}+1})}\\
    &\quad
    \leq
    \n{\sp{\ael,\varepsilon \nabla \ael, \vel}}_{\widetilde{L^r}(0,\infty;\dB_{q,1}^{\frac{2}{q}-1+\frac{2}{r}})}
    +
    CA[a_0,v_0]
    e^{CA[a_0,v_0]}
\n{\sp{\vel,w}}_{L^1(0,\infty;\fB_{p,1}^{\frac{2}{p}+1})}^{h;\frac{1}{\varepsilon}}\\
    &\qquad 
    +
    CA[a_0,v_0]
    e^{CA[a_0,v_0]}
    {\n{\sp{\ael,\vel}}
    _{
    L^4(0,\infty;\dB_{\theta,1}^{\frac{2}{\theta}-\frac{1}{2}})
    \cap 
    L^{\frac{4}{3}}(0,\infty;\dB_{\theta,1}^{\frac{2}{\theta}+\frac{1}{2}})}}\\
    &\quad
    \to 0\qquad{\rm as}\quad \varepsilon \downarrow 0,
\end{align}
which is implied by Lemma \ref{lemm:lin-0-2}, \eqref{AV} and \eqref{vel-high}.
Hence, we complete the proof.  
\end{proof}

\vspace{10mm}
{\bf{Acknowledgements}} 
M. Fujii was supported by Grant-in-Aid for Research Activity Start-up, Grant Number JP23K19011.
Y. Li is partially supported by National Natural Science Foundation of China under grant number 12001003; he sincerely thanks Professor Yongzhong Sun for patient guidance and encouragement.

\vspace{4mm}
{\bf{Data Availability}} Data sharing is not applicable to this article as no datasets were generated or analyzed
during the current study.
\vspace{4mm}

{\bf{Conflicts of interest}} All authors certify that there are no conflicts of interest for this work.

\appendix
\def\thesection{\Alph{section}}
\section{Proof of Proposition \ref{prop:lim}}\label{sec:a}
The aim of this section is to prove Proposition \ref{prop:lim}.
To begin with, we the following two propositions. 
\begin{lemm}\label{lemm:lim-loc-1}
    Let $1\leq p \leq \infty$.
    Then, there exists a positive constant $\delta_0=\delta_0(p,\sigma)$ such that if $\n{w_0}_{\fB_{p,\sigma}^{\frac{2}{p}-1}} \leq \delta_0$, then the equations \eqref{eq:lim-2} possess a solution 
    \begin{align}
        w \in \widetilde{C}([0,\infty);\fB_{p,\sigma}^{\frac{2}{p}-1}(\mathbb{R}^2)) \cap \widetilde{L^1}(0,\infty;\fB_{p,\sigma}^{\frac{2}{p}+1}(\mathbb{R}^2)).
    \end{align}
    Moreover, there exists a positive constant $C=C(p,\sigma)$ such that  
    \begin{align}
        \n{w}_{\widetilde{L^{\infty}}(0,\infty;\fB_{p,\sigma}^{\frac{2}{p}-1}) \cap \widetilde{L^1}(0,\infty;\fB_{p,\sigma}^{\frac{2}{p}+1})}
        \leq
        C
        \n{w_0}_{\fB_{p,\sigma}^{\frac{2}{p}-1}}.
    \end{align}
\end{lemm}
\begin{lemm}\label{lemm:lim-loc-2}
    Let $1\leq p \leq \infty$.
    Then, for any divergence free vector field $w_0 \in \fB_{p,1}^{\frac{2}{p}-1}(\mathbb{R}^2)$ and divergence free vector field $v \in \widetilde{L^{\infty}}(0,\infty;\fB_{p,1}^{\frac{2}{p}-1}(\mathbb{R}^2)) \cap \widetilde{L^1}(0,\infty;\fB_{p,1}^{\frac{2}{p}+1}(\mathbb{R}^2))$, there exists a positive time $T_0=T_0(p,1,w_0,v)$ such that there exists a local solution 
    \begin{align}
        w \in \widetilde{C}([0,T_0);\fB_{p,1}^{\frac{2}{p}-1}(\mathbb{R}^2)) \cap \widetilde{L^1}(0,T_0;\fB_{p,1}^{\frac{2}{p}+1}(\mathbb{R}^2))
    \end{align} 
    to the equations
    \begin{align}\label{w-perturb}
        \begin{cases}
            \partial_t w - \mu \Delta w + \mathbb{P}(v \cdot \nabla)w + \mathbb{P}(w \cdot \nabla)v + \mathbb{P}(w \cdot \nabla)w = 0,\\
            \div w = 0,\\
            w(0,x) = w_0(x).
        \end{cases}
    \end{align}
    Moreover, if we additionally assume $2 \leq p \leq 4$ and $w_0 \in L^2(\mathbb{R}^2)$, then it holds 
    \begin{align}
        w \in C([0,T_0);L^2(\mathbb{R}^2)) \cap L^2(0,T_0;\dH^1(\mathbb{R}^2)).
    \end{align}
\end{lemm}

\begin{proof}[Proofs of Lemmas \ref{lemm:lim-loc-1} and \ref{lemm:lim-loc-2}]
We provide the outline of the proofs of Lemmas \ref{lemm:lim-loc-1} and \ref{lemm:lim-loc-2} briefly.
Let us define a complete metric space
\begin{align}
    X^p(\delta,T)
    :={}&
    \Mp{ 
    \begin{aligned}
    u 
    \in{}& 
    \widetilde{C}([0,T);\fB_{p,1}^{\frac{2}{p}-1}(\mathbb{R}^2))\\ 
    &\cap L^1(0,T;\fB_{p,1}^{\frac{2}{p}+1}(\mathbb{R}^2))
    \end{aligned}
    \ ; \ 
    \n{u}_{\widetilde{L^2}(0,T;\fB_{p,1}^{\frac{2}{p}})}
    \leq \delta 
    },\\
    d_X(u_1,u_2):={}&
    \n{u^1-u^2}_{\widetilde{L^2}(0,T;\fB_{p,1}^{\frac{2}{p}})}.
\end{align}
Then, since it follows from Lemma \ref{lemm:max-heat} and the nonlinear estimate \eqref{prod-est-B} that 
\begin{align}
    &
    \n{\int_0^t e^{(t-\tau)\Delta} \mathbb{P} \div (u^1(\tau) \otimes u^2(\tau))d\tau}_{\widetilde{L^{\infty}}(0,T;\fB_{p,1}^{\frac{2}{p}-1}) \cap L^1(0,T;\fB_{p,1}^{\frac{2}{p}+1})}\\
    &\quad 
    \leq
    C
    \n{u^1\otimes u^2}_{L^1(0,T;\fB_{p,1}^{\frac{2}{p}})}\\
    &\quad 
    \leq
    C
    \n{u^1}_{L^2(0,T;\fB_{p,1}^{\frac{2}{p}})}
    \n{u^2}_{L^2(0,T;\fB_{p,1}^{\frac{2}{p}})},
\end{align}
the standard contraction mapping argument enables us to obtain the local well-posedness of \eqref{w-perturb} for large data or global well-posedness of \eqref{eq:lim-2} for small data in $\fB_{p,1}^{\frac{2}{p}-1}(\mathbb{R}^2)$. 

Finally, we show the $L^2$-regularity of the solution to \eqref{w-perturb} stated in Lemma \ref{lemm:lim-loc-2}.
By Lemma \ref{lemm:max-heat} via $\dB_{2,2}^s(\mathbb{R}^2) = \dot{H}^s(\mathbb{R}^2)$ ($s \in \mathbb{R}$), the solution $w$ to \eqref{w-perturb} satisfies
\begin{align}
    \sup_{0\leq t \leq T_0}\n{w(t)}_{L^2}
    +\n{w}_{L^2(0,T_0;\dot{H}^1)}
    \leq{}&
    C
    \n{w}_{\widetilde{L^{\infty}}(0,T_0;\dB_{2,2}^0) \cap \widetilde{L^2}(0,T_0;\dB_{2,2}^1)}\\
    \leq{}&
    C\n{w_0}_{L^2}
    +
    C\n{v \otimes w}_{\widetilde{L^1}(0,T_0;\dB_{2,2}^{0})}\\
    &
    +
    C\n{w \otimes v}_{\widetilde{L^1}(0,T_0;\dB_{2,2}^{0})}
    +
    C\n{w \otimes w}_{\widetilde{L^1}(0,T_0;\dB_{2,2}^{0})}.
\end{align}
Here, we see by \eqref{est-f-grad-g} that  
\begin{align} 
    &
    \n{v \otimes w}_{\widetilde{L^1}(0,T_0;\dB_{2,2}^{1})}
    +
    \n{w \otimes v}_{\widetilde{L^1}(0,T_0;\dB_{2,2}^{1})}
    +
    \n{w \otimes w}_{\widetilde{L^1}(0,T_0;\dB_{2,2}^{1})}\\
    &\quad 
    \leq
    C
    \n{v}_{L^4(0,T;\dB_{p,1}^{\frac{2}{p}-\frac{1}{2}})}
    \n{w}_{L^{\frac{4}{3}}(0,T;\dB_{p,1}^{\frac{2}{p}+\frac{1}{2}})}
    +
    C
    \n{w}_{L^{\frac{4}{3}}(0,T;\dB_{p,1}^{\frac{2}{p}+\frac{1}{2}})}
    \n{v}_{L^4(0,T;\dB_{p,1}^{\frac{2}{p}-\frac{1}{2}})}\\
    &\qquad
    +
    C
    \n{w}_{L^{\frac{4}{3}}(0,T;\dB_{p,1}^{\frac{2}{p}+\frac{1}{2}})}
    \n{w}_{L^4(0,T;\dB_{p,1}^{\frac{2}{p}-\frac{1}{2}})}\\
    &\quad 
    \leq
    C
    \n{v}_{L^4(0,T;\fB_{p,1}^{\frac{2}{p}-\frac{1}{2}})}
    \n{w}_{L^{\frac{4}{3}}(0,T;\fB_{p,1}^{\frac{2}{p}+\frac{1}{2}})}
    +
    C
    \n{w}_{L^{\frac{4}{3}}(0,T;\fB_{p,1}^{\frac{2}{p}+\frac{1}{2}})}
    \n{v}_{L^4(0,T;\fB_{p,1}^{\frac{2}{p}-\frac{1}{2}})}\\
    &\qquad
    +
    C
    \n{w}_{L^{\frac{4}{3}}(0,T;\fB_{p,1}^{\frac{2}{p}+\frac{1}{2}})}
    \n{w}_{L^4(0,T;\fB_{p,1}^{\frac{2}{p}-\frac{1}{2}})}\\
    &\quad
    < \infty,
\end{align}
which and $w_0 \in L^2(\mathbb{R}^2)$ implies 
\begin{align}
    \sup_{0\leq t \leq T_0}\n{w(t)}_{L^2}
    +\n{w}_{L^2(0,T_0;\dot{H}^1)}
    <\infty.
\end{align}
This completes the proof.
\end{proof}

With Lemmas \ref{lemm:lim-loc-1} and \ref{lemm:lim-loc-2} at hand, we are ready to give the proof of Proposition \ref{prop:lim}. 

\begin{proof}[Proof of Proposition \ref{prop:lim}]
Let $0 < \delta \leq \delta_0$ be a positive constant to be determined later.
Then, since $L^2(\mathbb{R}^2)$ is dense in $\fB_{p,1}^{\frac{2}{p}-1}(\mathbb{R}^2)$, there exists a $w_{0}^{(1)} \in L^2(\mathbb{R}^2)$ such that 
\begin{align}
    \n{w_0 - w_0^{(1)}}_{\fB_{p,1}^{\frac{2}{p}-1}}\leq\delta.
\end{align}
Choosing $\delta$ sufficiently small, it follows from Lemma \ref{lemm:lim-loc-1} that \eqref{eq:lim-2} with the initial data $w_0^{(2)}:=w_0 - w_0^{(1)}$ possesses a unique solution
\begin{align}
    w^{(2)} 
    \in 
    \widetilde{C}([0,\infty);\fB_{p,1}^{\frac{2}{p}-1}(\mathbb{R}^2)) \cap {L^1}(0,\infty;\fB_{p,1}^{\frac{2}{p}+1}(\mathbb{R}^2))
\end{align}
with the estimate
\begin{align}
    \n{w^{(2)} }_{\widetilde{L^{\infty}}(0,\infty;\fB_{p,1}^{\frac{2}{p}-1}) \cap {L^1}(0,\infty;\fB_{p,1}^{\frac{2}{p}+1})}
    \leq
    C
    \n{w_0^{(2)} }_{\fB_{p,1}^{\frac{2}{p}-1}}
    \leq
    C\delta.
\end{align}
Lemma \ref{lemm:lim-loc-2} ensures the existence of a positive time $T_0$ and the solution  
\begin{align}
    w^{(1)} 
    \in{}
    &
    \widetilde{C}([0,T_0);\fB_{p,1}^{\frac{2}{p}-1}(\mathbb{R}^2)) \cap {L^1}(0,T_0;  \fB_{p,1}^{\frac{2}{p}+1}(\mathbb{R}^2))\\
    &\cap 
    C([0,T_0);L^2(\mathbb{R}^2)) \cap L^2(0,T_0;\dH^1(\mathbb{R}^2)).
\end{align}
of the equations 
\begin{align}
    \begin{cases}
        \partial_t w^{(1)} - \mu \Delta w^{(1)} + \mathbb{P}(w^{(2)} \cdot \nabla)w^{(1)} + \mathbb{P}(w^{(1)} \cdot \nabla)w^{(2)} + \mathbb{P}(w^{(1)} \cdot \nabla)w^{(1)} = 0,\\
        \div w^{(1)} = 0,\\
        w^{(1)}(0,x) = w_0^{(1)}(x)
    \end{cases}
\end{align}
Then, it is easy to see that  
\begin{align}
    w:=w^{(1)} + w^{(2)}
    \in \widetilde{C}([0,T_0);\fB_{p,1}^{\frac{2}{p}-1}(\mathbb{R}^2)) \cap {L^1}(0,T_0;\fB_{p,1}^{\frac{2}{p}+1}(\mathbb{R}^2)).
\end{align}
solves \eqref{eq:lim-2} on the time interval $[0,T_0)$.
To complete the proof, it suffices to establish a global a priori estimate of $w$ on the time interval $[0,T)$, where $T$ is an arbitrary time not exceeding the maximal existence time.
By the energy method, we have 
\begin{align}
    \frac{1}{2}\frac{d}{dt}\n{w^{(1)}(t)}_{L^2}^2
    +
    \mu
    \n{\nabla w^{(1)}(t)}_{L^2}^2
    ={}&
    -\left\langle (w^{(1)}(t) \cdot \nabla)w^{(2)}(t), w^{(1)}(t) \right \rangle_{L^2}\\
    ={}&
    \left\langle (w^{(1)}(t) \cdot \nabla)w^{(1)}(t), w^{(2)}(t) \right \rangle_{L^2}\\
    \leq{}&
    \n{(w^{(1)}(t) \cdot \nabla)w^{(1)}(t)}_{ L^1       } 
    \n{w^{(2)}(t)}_{ L^{\infty}   }\\
    \leq{}& 
    C
    \n{w^{(1)}(t)}_{L^2}
    \n{\nabla w^{(1)}(t)}_{L^2}
    \n{w^{(2)}(t)}_{\fB_{p,1}^{\frac{2}{p}}}.
\end{align}
Thus, we see that 
\begin{align}
    &\frac{1}{2}\n{w^{(1)}(t)}_{L^2}^2
    +
    \mu
    \n{w^{(1)}}_{L^2(0,t;\dH^1)}^2\\
    &\quad 
    \leq
    \n{w_0^{(1)}}_{L^2}^2
    +
    C
    \n{w^{(1)}}_{L^{\infty}(0,t;L^2)}
    \n{w^{(1)}}_{L^2(0,t;\dH^1)}
    \n{w^{(2)}}_{L^2(0,t;\fB_{p,1}^{\frac{2}{p}})}\\
    &\quad 
    \leq
    \n{w_0^{(1)}}_{L^2}^2
    +
    C_1\delta
    \n{w^{(1)}}_{L^{\infty}(0,T;L^2)}
    \n{w^{(1)}}_{L^2(0,T;\dH^1)}.
\end{align}
for some positive constant $C_1$. 
We assume that $0<\delta \leq \sqrt{\mu}/(2C_1)$.
Then, we see that 
\begin{align}\label{ene:w^(1)-1}
    \n{w^{(1)}}_{L^{\infty}(0,T;L^2)}
    +
    \mu
    \n{\nabla w^{(1)}}_{L^2(0,T;\dH^1)}
    \leq
    C
    \n{w_0^{(1)}}_{L^2}.
\end{align}
Therefore, from the real interpolation
\begin{align}
    L^4(0,T;\dB_{2,1}^{\frac{1}{2}}(\mathbb{R}^2))
    =
    \sp{
    L^{\infty}(0,T;L^2(\mathbb{R}^2)),
    L^2(0,T;\dH^1(\mathbb{R}^2))
    }_{\frac{1}{2},1}
\end{align}
and the estimate \eqref{ene:w^(1)-1}, we have 
\begin{align}\label{ene:w^(1)-2}
    \int_0^T
    \n{w^{(1)}(t)}_{\dB_{2,1}^{\frac{1}{2}}}^4
    dt
    \leq 
    C\n{w_0^{(1)}}_{L^2}^4.
\end{align}
Now, we estimate $w$.
It follows from the maximal regularity estimates of the heat kernel that 
\begin{align}
    &
    \n{w}_{\widetilde{L^{\infty}}(0,T;\fB_{p,1}^{\frac{2}{p}-1})}
    +
    \int_0^T
    \n{w(t)}_{\fB_{p,1}^{\frac{2}{p}+1}}
    dt\\
    &\quad
    \leq
    C\n{w_0}_{\fB_{p,1}^{\frac{2}{p}-1}}
    +
    C\int_0^T
    \n{(w(t)\cdot\nabla) w(t)}_{\fB_{p,1}^{\frac{2}{p}-1}}
    dt.
\end{align}
For the estimate of the nonlinear term, we see that 
\begin{align}
    &\int_0^T
    \n{(w(t)\cdot\nabla) w(t)}_{\fB_{p,1}^{\frac{2}{p}-1}}
    dt\\
    &\quad 
    \leq{}
    C
    \int_0^T
    \n{w(t)\otimes w(t)}_{\fB_{p,1}^{\frac{2}{p}}}
    dt\\
    &\quad 
    \leq{}
    C
    \int_0^T
    \n{w(t)}_{\fB_{p,1}^{\frac{2}{p}-\frac{1}{2}}}
    \n{ w(t)}_{\fB_{p,1}^{\frac{2}{p}+\frac{1}{2}}}
    dt\\
    &\quad 
    \leq{}
    C
    \int_0^T
    \n{w(t)}_{\fB_{p,1}^{\frac{2}{p}-\frac{1}{2}}}
    \n{ w(t)}_{\fB_{p,1}^{\frac{2}{p}-1}}^{\frac{1}{4}}
    \n{ w(t)}_{\fB_{p,1}^{\frac{2}{p}+1}}^{\frac{3}{4}}
    dt\\
    &\quad 
    \leq{}
    C
    \int_0^T
    \n{w(t)}_{\fB_{p,1}^{\frac{2}{p}-\frac{1}{2}}}^4
    \n{ w}_{\widetilde{L^{\infty}}(0,t;\fB_{p,1}^{\frac{2}{p}-1})}
    dt
    +
    \frac{1}{2}
    \int_0^T
    \n{ w(t)}_{\fB_{p,1}^{\frac{2}{p}+1}}
    dt,
\end{align}
where we have used the interpolation inequality 
\begin{align}
    \n{ w(t) }_{\fB_{p,1}^{\frac{2}{p}+\frac{1}{2}}}
    \leq{}
    \n{ w(t) }_{\fB_{p,1}^{\frac{2}{p}-1}}^{\frac{1}{4}}
    \n{ w(t) }_{\fB_{p,1}^{\frac{2}{p}+1}}^{\frac{3}{4}}.
\end{align}
Thus, we obtain by Lemma \ref{lemm:max-heat} that
\begin{align}
    &
    \n{w}_{\widetilde{L^{\infty}}(0,T;\fB_{p,1}^{\frac{2}{p}-1})}
    +
    \int_0^T
    \n{w(t)}_{\fB_{p,1}^{\frac{2}{p}+1}}
    dt\\
    &\quad 
    \leq
    C\n{w_0}_{\fB_{p,1}^{\frac{2}{p}-1}}
    +
    C
    \int_0^T
    \n{w(t)}_{\fB_{p,1}^{\frac{2}{p}-\frac{1}{2}}}^4
    \n{ w}_{\widetilde{L^{\infty}}(0,t;\fB_{p,1}^{\frac{2}{p}-1})}
    dt,
\end{align}
which and the Gronwall inequality imply
\begin{align}
    \n{w}_{\widetilde{L^{\infty}}(0,T;\fB_{p,1}^{\frac{2}{p}-1})}
    +
    \int_0^T
    \n{w(t)}_{\fB_{p,1}^{\frac{2}{p}+1}}
    dt
    \leq
    C\n{w_0}_{\fB_{p,1}^{\frac{2}{p}-1}}
    \exp
    \lp{ 
    \int_0^T
    \n{w(t)}_{\fB_{p,1}^{\frac{2}{p}-\frac{1}{2}}}^4
    dt}.
\end{align}
Here, by \eqref{ene:w^(1)-2}, it holds
\begin{align}
    \int_0^T
    \n{w(t)}_{\fB_{p,1}^{\frac{2}{p}-\frac{1}{2}}}^4
    dt
    \leq{}&
    C
    \int_0^T
    \n{w^{(1)}(t)}_{\fB_{p,1}^{\frac{2}{p}-\frac{1}{2}}}^4
    dt
    +
    C
    \int_0^T
    \n{w^{(2)}(t)}_{\fB_{p,1}^{\frac{2}{p}-\frac{1}{2}}}^4
    dt\\
    \leq{}&
    C
    \int_0^T
    \n{w^{(1)}(t)}_{\fB_{2,1}^{\frac{1}{2}}}^4
    dt\\
    &
    +
    C
    \n{w^{(2)}}_{\widetilde{L^{\infty}}(0,\infty;\fB_{p,\sigma}^{\frac{2}{p}-1})}^3
    \n{w^{(2)}}_{\widetilde{L^1}(0,\infty;\fB_{p,\sigma}^{\frac{2}{p}+1})}\\
    \leq{}&
    C\n{w_0^{(1)}}_{L^2}^4
    +
    C\delta^4.
\end{align}
Hence, we obtain 
\begin{align}
    \n{w}_{\widetilde{L^{\infty}}(0,T;\fB_{p,1}^{\frac{2}{p}-1})}
    +
    \int_0^T
    \n{w(t)}_{\fB_{p,1}^{\frac{2}{p}+1}}
    dt
    \leq
    C\n{w_0}_{\fB_{p,1}^{\frac{2}{p}-1}}
    \exp
    \lp{ 
    C\n{w_0^{(1)}}_{L^2}^4
    +
    C\delta^4},
\end{align}
which completes the proof.
\end{proof}

\section{Proof of Proposition \ref{lemm:LWP-CNSK}}\label{sec:LWP}
We show Proposition \ref{lemm:LWP-CNSK}.
\begin{proof}[Proof of Proposition \ref{lemm:LWP-CNSK}]
Let $T>0$ and $0<\varepsilon\leq 1$.
By the Duhamel principle the system \eqref{eq:NSK-2} is rewritten as 
$(\ae,\ve)=\Phi^{\varepsilon}[\ae,\ve]$.
Here, we have defined the map $\Phi^{\varepsilon}$ by  
\begin{align}
    \Phi^{\varepsilon} 
    [a,v](t)
    :=
    \mathcal{G}^{\varepsilon}(t)
    \begin{pmatrix}
        a_0 \\ v_0
    \end{pmatrix}
    -
    \int_0^t
    \mathcal{G}^{\varepsilon}(t-\tau)
    \begin{pmatrix}
        \mathcal{A}[a(\tau),v(\tau)] \\ 
        \mathcal{V}_{\varepsilon}[a(\tau),v(\tau)]
    \end{pmatrix}
    d\tau
\end{align}
for $(a,v) \in Z_T^{\varepsilon}$, where
$\Mp{\mathcal{G}^{\varepsilon}(t)}_{t>0}$ is the semigroup corresponding to the linear problem
\begin{align}\label{eq:NSK-lin-G}
    \begin{dcases}
        \partial_t \ae + \dfrac{1}{\varepsilon} \div \ve = 0, \\
        \partial_t \ve - \mathcal{L} \ve
        -\kappa \Delta (\varepsilon \nabla \ae)
        +\frac{1}{\varepsilon}\nabla \ae
        = 
        0
    \end{dcases}
\end{align}
and function space $Z_T^{\varepsilon}$ is defined by
\begin{align}
    &
    Z_T^{\varepsilon}
    :=
    \Mp{
    (a,v)
    \ ;\ 
    \begin{aligned}
    &
    (a,\varepsilon \nabla a, v)
    \in 
    {}
    C([0,T];\fB_{p,1}^{\frac{2}{p}-1}(\mathbb{R}^2))
    \cap 
    L^1(0,T;\fB_{p,1}^{\frac{2}{p}+1}(\mathbb{R}^2))\\
    &
    \n{(a,v)}_{Z_T^{\varepsilon}}
    \leq
    \frac{1}{2C_0}
    \end{aligned}
    },\\
    &
    \begin{aligned}
    \n{(a,v)}_{Z_T^{\varepsilon}}
    :={}
    &
    \n{a}_{{L^2}(0,T;\fB_{p,1}^{\frac{2}{p}})}
    +
    \n{\varepsilon \nabla a}_{L^{\infty}(0,T;\fB_{p,1}^{\frac{2}{p}-1}) \cap {L^2}(0,T;\fB_{p,1}^{\frac{2}{p}})}\\
    &
    +
    \n{v}_{{L^2}(0,T;\fB_{p,1}^{\frac{2}{p}})\cap {L^1}(0,T;\fB_{p,1}^{\frac{2}{p}+1})},
    \end{aligned}
\end{align} 
with a positive constant $C_0:=\max\{C_1,2C_2\}$, where $C_1$ and $C_2$ are to be determined later.
It follows from Lemma \ref{lemm:lin-0-1} and 
\begin{align}
    &
    \begin{aligned}
        \n{\div(av)}_{L^1(0,T;\fB_{p,1}^{\frac{2}{p}-1})}
        \leq
        C
        \n{a}_{L^2(0,T;\fB_{p,1}^{\frac{2}{p}})}
        \n{v}_{L^2(0,T;\fB_{p,1}^{\frac{2}{p}})},
    \end{aligned}
    \\
    &
    \begin{aligned}
        \n{\varepsilon \nabla \div(av)}_{L^1(0,T;\fB_{p,1}^{\frac{2}{p}-1})}
        \leq{}&
        C
        \n{\varepsilon \nabla a}_{L^{\infty}(0,T;\fB_{p,1}^{\frac{2}{p}-1})}
        \n{v}_{L^1(0,T;\fB_{p,1}^{\frac{2}{p}+1})}\\
        &
        +
        C
        \n{\varepsilon \nabla a}_{L^2(0,T;\fB_{p,1}^{\frac{2}{p}})}
        \n{v}_{L^2(0,T;\fB_{p,1}^{\frac{2}{p}})},
    \end{aligned}
    \\
    &
    \begin{aligned}
        \n{(v \cdot \nabla)v}_{L^1(0,T;\fB_{p,1}^{\frac{2}{p}-1})}
        \leq{}&
        C
        \n{v}_{L^2(0,T;\fB_{p,1}^{\frac{2}{p}})}^2,
    \end{aligned}
    \\
    &
    \begin{aligned}
        \n{\mathcal{J}(\varepsilon a)\mathcal{L}v}_{L^1(0,T;\fB_{p,1}^{\frac{2}{p}-1})}
        \leq{}&
        C
        \n{\varepsilon \nabla a}_{L^{\infty}(0,T;\fB_{p,1}^{\frac{2}{p}-1})}
        \n{v}_{L^1(0,T;\fB_{p,1}^{\frac{2}{p}+1})},
    \end{aligned}
    \\
    &
    \begin{aligned}
        \n{\frac{1}{\varepsilon}\mathcal{K}(\varepsilon a)\nabla a}_{L^1(0,T;\fB_{p,1}^{\frac{2}{p}-1})}
        \leq{}&
        C
        \n{a}_{L^2(0,T;\fB_{p,1}^{\frac{2}{p}})}^2
    \end{aligned}
\end{align}
that 
\begin{align}
    \n{\Phi^{\varepsilon}[a,v]}_{Z_T^{\varepsilon}}
    \leq
    \n{
    (\ael,\vel)
    }_{Z_T^{\varepsilon}}
    +
    C_0
    \n{(a,v)}_{Z_T^{\varepsilon}}^2
\end{align}
for some positive constant $C_0$ independent of $\varepsilon$ and $T$, where we have set 
\begin{align}
    \begin{pmatrix}
        \ael(t) \\ \vel(t)
    \end{pmatrix}
    :=
    \mathcal{G}^{\ep}(t)  
    \begin{pmatrix}
        a_0 \\ v_0
    \end{pmatrix}.
\end{align} 
By virtue of Lemma \ref{lemm:lin-0-1}, there exists a positive constant $0<\varepsilon_0 \leq 1$ such that 
\begin{align}
    \n{\varepsilon \nabla \ael}_{L^{\infty}(0,\infty;\fB_{p,1}^{\frac{2}{p}-1})}
    \leq
    \frac{1}{8C_0}
\end{align}
for all $0<\varepsilon \leq \varepsilon_0$.
We fix such $\varepsilon$ in the following of this proof.
From the estimate 
\begin{align}
    \n{\Delta_j(\ael, \varepsilon \nabla \ael,\vel)}_{\widehat{L^p}}
    \leq
    C
    e^{-c2^{2j}t}
    \n{\Delta_j(a_0, \nabla a_0,v_0)}_{\widehat{L^p}}
\end{align}
via the Lebesgue dominated convergence theorem,
there exists a positive time $T_0$ such that 
\begin{align}
    \n{(\ael, \varepsilon \nabla \ael)}_{{L^2}(0,T_0;\fB_{p,1}^{\frac{2}{p}})}
    +
    \n{\vel}_{{L^2}(0,T_0;\fB_{p,1}^{\frac{2}{p}})\cap {L^1}(0,T;\fB_{p,1}^{\frac{2}{p}+1})}
    \leq
    \frac{1}{8C_0}.
\end{align}
Then, we see that 
\begin{align}
    \n{(\ael,\vel)}_{Z_{T_0}^{\varepsilon}}
    \leq
    \frac{1}{4C_0}.
\end{align}
Thus, we see that 
\begin{align}
    \n{\Phi^{\varepsilon}[a,v]}_{Z_{T_0}^{\varepsilon}}
    &\leq
    \n{
    (\ael,\vel)
    }_{Z_{T_0}^{\varepsilon}}
    +
    C_0
    \n{(a,v)}_{Z_{T_0}^{\varepsilon}}^2\\
    &\leq
    \frac{1}{4C_0}
    +
    C_0
    \sp{\frac{1}{2C_0}}^2
    =
    \frac{1}{2C_0},
\end{align}
which implies $\Phi[a,v] \in Z_{T_0}^{\varepsilon}$.
Similarly to above, we see that 
\begin{align}
    \n{\Phi^{\varepsilon}[a,v]-\Phi^{\varepsilon}[b,u]}_{Z_{T_0}^{\varepsilon}}
    &\leq
    C_2
    \sp{\n{(a,v)}_{Z_{T_0}^{\varepsilon}}+\n{(b,u)}_{Z_{T_0}^{\varepsilon}}}
    \n{(a-b,v-u)}_{Z_{T_0}^{\varepsilon}}
\end{align}
for all $(a,v),(b,u) \in Z_{T_0}^{\varepsilon}$ some positive constant $C_2$. 
Then, it holds from the choice of $C_0$ that 
\begin{align}
    \n{\Phi^{\varepsilon}[a,v]-\Phi^{\varepsilon}[b,u]}_{Z_{T_0}^{\varepsilon}}
    &\leq
    C_2
    \sp{\n{(a,v)}_{Z_{T_0}^{\varepsilon}}+\n{(b,u)}_{Z_{T_0}^{\varepsilon}}}
    \n{(a-b,v-u)}_{Z_{T_0}^{\varepsilon}}\\
    &\leq
    \frac{C_0}{2}
    \sp{\frac{1}{2C_0}+\frac{1}{2C_0}}
    \n{(a-b,v-u)}_{Z_{T_0}^{\varepsilon}}\\
    &
    =
    \frac{1}{2}
    \n{(a-b,v-u)}_{Z_{T_0}^{\varepsilon}}.
\end{align}
Hence, the contraction mapping principle implies there exists a unique $(\ae,\ve) \in Z_{T_0}^{\varepsilon}$ such that $(\ae,\ve)=\Phi[\ae,\ve]$, which completes the proof.
\end{proof}

\end{document}